\documentclass[journal]{IEEEtran}
\usepackage{cite}

\ifCLASSINFOpdf
  \usepackage[pdftex]{graphicx}
\else
  \usepackage[dvips]{graphicx}
\fi
\usepackage[cmex10]{amsmath}
\usepackage[tight,footnotesize]{subfigure}
\usepackage{url}

\usepackage{psfrag}
\usepackage{amssymb}  
\usepackage[amsmath,thmmarks]{ntheorem}
\usepackage{multirow}
\usepackage{textcomp}
\usepackage{slashbox}
\usepackage{xcolor}
\usepackage{algorithmic}
\usepackage{algorithm}

\theorembodyfont{\rmfamily} \theoremstyle{plain}

\newtheorem{proposition}{Proposition}
\newtheorem{assumption}{Assumption}
\newtheorem{remark}{Remark}

\newtheorem{problem}{Problem}

\newcommand{\subscr}[2]{#1_{\textup{#2}}}
\newcommand{\subpscr}[3]{#1_{\textup{#2}}^{\textup{#3}}}

\begin{document}
%
\title{Storage Size Determination for Grid-Connected Photovoltaic Systems}
%
%
%

\author{Yu Ru, Jan Kleissl, and Sonia Martinez
\thanks{Yu Ru, Jan Kleissl, and Sonia Martinez are with the Mechanical and Aerospace Engineering Department, University of California, San Diego (e-mail: {\tt\small yuru2@ucsd.edu, jkleissl@ucsd.edu, soniamd@ucsd.edu}).}}

\maketitle

\begin{abstract}
In this paper, we study the problem of determining the size of battery storage used in grid-connected photovoltaic (PV) systems. In our setting, electricity is generated from PV and is used to supply the demand from loads. Excess electricity generated from the PV can be stored in a battery to be used later on, and electricity must be purchased from the electric grid if the PV generation and battery discharging cannot meet the demand. Due to the time-of-use electricity pricing, electricity can also be purchased from the grid when the price is low, and be sold back to the grid when the price is high. The objective is to minimize the cost associated with purchasing from (or selling back to) the electric grid and the battery capacity loss while at the same time satisfying the load and reducing the peak electricity purchase from the grid. Essentially, the objective function depends on the chosen battery size. We want to find a unique critical value (denoted as $\subpscr{C}{ref}{c}$) of the battery size such that the total cost remains the same if the battery size is larger than or equal to $\subpscr{C}{ref}{c}$, and the cost is strictly larger if the battery size is smaller than $\subpscr{C}{ref}{c}$. We obtain a criterion for evaluating the economic value of batteries compared to purchasing electricity from the grid, propose lower and upper bounds on $\subpscr{C}{ref}{c}$, and introduce an  efficient algorithm for calculating its value; these results are validated via simulations.
\end{abstract}
%
%
%
%
%
%
%
%
%
\section{Introduction} \label{section1}

The need to reduce greenhouse gas emissions due to fossil fuels and the liberalization of the electricity market have led to large scale development of renewable energy generators in electric grids~\cite{yru_journal:Kanchev_2011}. Among renewable energy technologies such as hydroelectric, photovoltaic (PV), wind, geothermal, biomass, and tidal systems, grid-connected solar PV continued to be the fastest growing power generation technology, with a $70$\% increase in existing capacity to $13$GW in 2008~\cite{yru_journal:Teleke_2010}. However, solar energy generation tends to be variable due to the diurnal cycle of the solar geometry and clouds. Storage devices (such as batteries, ultracapacitors, compressed air, and pumped hydro storage~\cite{yru_journal:Lafoz_2008}) can be used to i) smooth out the fluctuation of the PV output fed into electric grids (``capacity firming")~\cite{yru_journal:Teleke_2010, yru_journal:Omran_2011}, ii) discharge and augment the PV output during times of peak energy usage (``peak shaving")~\cite{yru_journal:Riffonneau_2011}, or iii) store energy for nighttime use, for example in zero-energy buildings and residential homes.

Depending on the specific application (whether it is off-grid or grid-connected), battery storage size is determined based on the battery specifications such as the battery storage capacity (and the minimum battery charging/discharging time). For off-grid applications, batteries have to fulfill the following requirements: (i) the discharging rate has to be larger than or equal to the peak load capacity; (ii) the battery storage capacity has to be large enough to supply the largest night time energy use and to be able to supply energy during the longest cloudy period (autonomy). The IEEE standard~\cite{yru_journal:IEEE_1013_2007} provides sizing recommendations for lead-acid batteries in stand-alone PV systems. In~\cite{yru_journal:Shrestha_1998}, the solar panel size and the battery size have been selected via simulations to optimize the operation of a stand-alone PV system, which considers reliability measures in terms of loss of load hours, the energy loss and the total cost. In contrast, if the PV system is grid-connected, autonomy is a secondary goal; instead, batteries can reduce the fluctuation of PV output or provide economic benefits such as demand charge reduction, and arbitrage. The work in~\cite{yru_journal:Akatsuka_2010} analyzes the relation between available battery capacity and output smoothing, and estimates the required battery capacity using simulations. In addition, the battery sizing problem has been studied for wind power applications~\cite{yru_journal:Wang_2008, yru_journal:Brekken_2011, yru_journal:Li_2011} and hybrid wind/solar power applications~\cite{yru_journal:Borowy_1996, yru_journal:Chedid_1997, yru_journal:Vrettos_2011}. In~\cite{yru_journal:Wang_2008}, design of a battery energy storage system is examined for the purpose of attenuating the effects of unsteady input power from wind farms, and solution to the problem via a computational procedure results in the determination of the battery energy storage system's capacity. Similarly, in~\cite{yru_journal:Li_2011}, based on the statistics of long-term wind speed data captured at the farm, a dispatch strategy is proposed which allows the battery capacity to be determined so as to maximize a defined service lifetime/unit cost index of the energy storage system; then a numerical approach is used due to the lack of an explicit mathematical expression to describe the lifetime as a function of the battery capacity. In~\cite{yru_journal:Brekken_2011}, sizing and control methodologies for a zinc-bromine flow battery-based energy storage system are proposed to minimize the cost of the energy storage system. However, the sizing of the battery is significantly impacted by specific control strategies.
In~\cite{yru_journal:Borowy_1996}, a methodology for calculating the
optimum size of a battery bank and the PV array for a stand-alone
hybrid wind/PV system is developed, and a simulation model is used to examine different combinations of the number of PV modules and the number
of batteries. In~\cite{yru_journal:Chedid_1997}, an approach is proposed to help designers determine the optimal design of a hybrid wind-solar
power system; the proposed analysis employs linear programming techniques to minimize the average production cost of electricity while meeting the load requirements in a reliable manner. In~\cite{yru_journal:Vrettos_2011}, genetic algorithms are used to optimally size the hybrid system components, i.e., to select the optimal wind turbine and PV rated power, battery energy storage system nominal capacity and inverter rating. The primary optimization objective is the minimization of the levelized energy cost of the island system over the entire lifetime of the project.

In this paper, we study the problem of determining the battery size for grid-connected PV systems. The targeted applications are primarily electricity customers with PV arrays ``behind the meter," such as residential and commercial buildings with rooftop PVs. In such applications, the objective is to reduce the energy cost and the loss of investment on storage devices due to aging effects while satisfying the loads and reducing peak electricity purchase from the grid (instead of smoothing out the fluctuation of the PV output fed into electric grids). Our setting\footnote{Note that solar panels and batteries both operate on DC, while the grid and loads operate on AC. Therefore, DC-to-AC and AC-to-DC power conversion is necessary when connecting solar panels and batteries with the grid and loads.} is shown in Fig.~\ref{fig1}. Electricity is generated from PV panels, and is used to supply different types of loads. Battery storage is used to either store excess electricity generated from PV systems for later use when PV generation is insufficient to serve the load, or purchase electricity from the grid when the time-of-use pricing is lower and sell back to the grid when the time-of-use pricing is higher. Without a battery, if the load was too large to be supplied by PV generated electricity, electricity would have to be purchased from the grid to meet the demand. Naturally, given the high cost of battery storage, the size of the battery storage should be chosen such that the cost of electricity purchase from the grid and the loss of investment on batteries are minimized. Intuitively, if the battery is too large, the electricity purchase cost could be the same as the case with a relatively smaller battery. In this paper, we show that there is a unique critical value (denoted as $\subpscr{C}{ref}{c}$, refer to Problem~\ref{problem:battery_size}) of the battery capacity such that the cost of electricity purchase and the loss of investments on batteries remains the same if the battery size is larger than or equal to $\subpscr{C}{ref}{c}$ and the cost is strictly larger if the battery size is smaller than $\subpscr{C}{ref}{c}$. We obtain a criterion for evaluating the economic value of batteries compared to purchasing electricity from the grid, propose lower and upper bounds on $\subpscr{C}{ref}{c}$ given the PV generation, loads, and the time period for minimizing the costs, and introduce an efficient algorithm for calculating the critical battery capacity based on the bounds; these results are validated via simulations.

The contributions of this work are the following: i) to the best of our knowledge, this is the first attempt on determining the battery size for grid-connected PV systems based on a theoretical analysis on the lower and upper bounds of the battery size; in contrast, most previous work are based on simulations, e.g., the work in~\cite{yru_journal:Borowy_1996, yru_journal:Wang_2008, yru_journal:Akatsuka_2010, yru_journal:Li_2011, yru_journal:Brekken_2011}; ii) a criterion for evaluating the economic value of batteries compared to purchasing electricity from the grid is derived (refer to Proposition~\ref{prop:K} and Assumption~\ref{assumption_price}), which can be easily calculated and could be potentially used for choosing appropriate battery technologies for practical applications; and iii) lower and upper bounds on the battery size are proposed, and an efficient algorithm is introduced to calculate its value for the given PV generation and dynamic loads; these results are then validated using simulations. Simulation results illustrate the benefits of employing batteries in grid-connected PV systems via peak shaving and cost reductions compared with the case without batteries (this is discussed in Section~\ref{section5}.B).

The paper is organized as follows. In the next section, we lay out our setting, and formulate the storage size determination problem. Lower and upper bounds on $\subpscr{C}{ref}{c}$ are proposed in Section~\ref{section3}. Algorithms are introduced in Section~\ref{section4} to calculate the value of the critical battery capacity. In Section~\ref{section5}, we validate the results via simulations. Finally, conclusions and future directions are given in Section~\ref{section6}.

\section{Problem Formulation} \label{section2}
In this section, we formulate the problem of determining the storage size for a grid-connected PV system, as shown in Fig.~\ref{fig1}. We first introduce different components in our setting.

\begin{figure}[b] \centering
\includegraphics[scale=0.6]{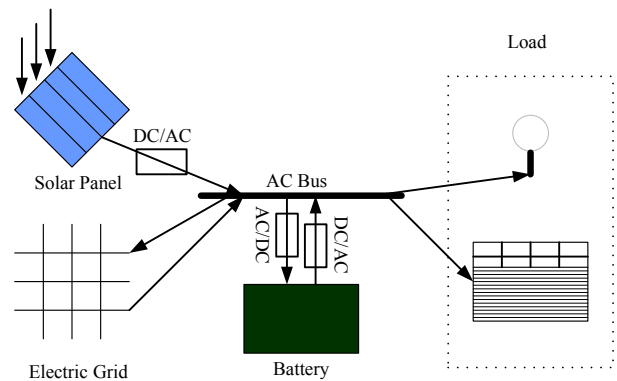}
\caption{Grid-connected PV system with battery storage and loads.}
\label{fig1}
\end{figure}

\subsection{Photovoltaic Generation}

We use the following equation to calculate the electricity generated from solar panels:
\begin{equation}
\subscr{P}{pv}(t) =   \textrm{GHI}(t) \times S \times \eta~, \label{eq:PV}
\end{equation}
where
\begin{itemize}
\item $\textrm{GHI}$ ($Wm^{-2}$) is the global horizontal irradiation at the location of solar panels,
\item $S$~($m^2$) is the total area of solar panels, and
\item $\eta$ is the solar conversion efficiency of the PV cells.
\end{itemize}
The PV generation model is a simplified version of the one used in~\cite{yru_journal:Rahman_2007} and does not account for PV panel temperature effects.\footnote{Note that our analysis on determining battery capacity only relies on $\subscr{P}{pv}(t)$ instead of detailed models of the PV generation. Therefore, more complicated PV generation models can also be incorporated into the cost minimization problem as discussed in Section~\ref{section2}.F.}

\subsection{Electric Grid}
Electricity can be purchased from or sold back to the grid. For simplicity, we assume that the prices for sales and purchases at time $t$ are identical and are denoted as $\subscr{C}{g}(t) (\$/Wh)$. Time-of-use pricing is used ($\subscr{C}{g}(t) \geq 0$ depends on $t$) because commercial buildings with PV systems that would consider a battery system usually pay time-of-use electricity rates. In addition, with increased deployment of smart meters and electric vehicles, some utility companies are moving towards residential time-of-use pricing as well; for example, SDG\&E (San Diego Gas \& Electric) has the peak, semipeak, offpeak prices for a day in the summer season~\cite{yru_journal:DynamicPricing}, as shown in Fig.~\ref{fig12}.

We use $\subscr{P}{g}(t) (W)$ to denote the electric power exchanged with the grid with the interpretation that
\begin{itemize}
\item $\subscr{P}{g}(t) > 0$ if electric power is purchased from the grid,
\item $\subscr{P}{g}(t) < 0$ if electric power is sold back to the grid.
\end{itemize}
In this way, positive costs are associated with the electricity purchase from the grid, and negative costs with the electricity sold back to the grid. In this paper, peak shaving is enforced through the constraint that $$\subscr{P}{g}(t) \leq D~,$$ where $D$ is a positive constant.

\subsection{Battery}
A battery has the following dynamic:
\begin{equation}
\frac{d E_B (t)}{d t} = P_B(t)~, \label{eq:battery_dynamic}
\end{equation}
where $E_B(t) (Wh)$ is the amount of electricity stored in the battery at time $t$, and $P_B(t) (W)$ is the charging/discharging rate; more specifically,
\begin{itemize}
\item $P_B(t) > 0$ if the battery is charging,
\item and $P_B(t) < 0$ if the battery is discharging.
\end{itemize}
For certain types of batteries, higher order models exist (e.g., a third order model is proposed in~\cite{yru_journal:Ceraolo_2000, yru_journal:Barsali_2002}).

To take into account battery aging, we use $C(t) (Wh)$ to denote the usable battery capacity at time $t$. At the initial time $t_0$, the usable battery capacity is $\subscr{C}{ref}$, i.e., $C(t_0) = \subscr{C}{ref} \geq 0$. The cumulative capacity loss at time $t$ is denoted as $\Delta C(t) (Wh)$, and $\Delta C(t_0) = 0$. Therefore, $$C(t) = \subscr{C}{ref} - \Delta C(t)~.$$ The battery aging satisfies the following dynamic equation
\begin{equation}
\frac{d \Delta C(t)}{d t} = \begin{cases} - Z P_B(t) & \mathrm{if~} P_B(t) < 0\\
0 & \mathrm{otherwise}~,
\end{cases} \label{eq:battery_aging}
\end{equation}
where $Z > 0$ is a constant depending on battery technologies. This aging model is derived from the aging model in~\cite{yru_journal:Riffonneau_2011} under certain reasonable assumptions; the detailed derivation is provided in the Appendix. Note that there is a capacity loss only when electricity is discharged from the battery. Therefore, $\Delta C(t)$ is a nonnegative and non-decreasing function of $t$.

We consider the following constraints on the battery:
\begin{itemize}
\item[i)] At any time, the battery charge $E_B(t)$ should satisfy $$0 \leq E_B(t) \leq C(t) = \subscr{C}{ref} - \Delta C(t)~,$$
\item[ii)] The battery charging/discharging rate should satisfy $$\subscr{P}{Bmin} \leq P_B(t) \leq \subscr{P}{Bmax}~,$$ where $\subscr{P}{Bmin}<0$, $-\subscr{P}{Bmin}$ is the maximum battery discharging rate, and $\subscr{P}{Bmax}>0$ is the maximum battery charging rate. For simplicity, we assume that $$\subscr{P}{Bmax} = -\subscr{P}{Bmin} = \frac{C(t)}{T_c} = \frac{\subscr{C}{ref} - \Delta C(t)}{T_c}~,$$ where constant $T_c > 0$ is the minimum time required to charge the battery from $0$ to $C(t)$ or discharge the battery from $C(t)$ to $0$.
\end{itemize}

\subsection{Load}
$\subscr{P}{load}(t) (W)$ denotes the load at time $t$. We do not make explicit assumptions on the load considered in Section~\ref{section3} except that $\subscr{P}{load}(t)$ is a (piecewise) continuous function. In residential home settings, loads could have a fixed schedule such as lights and TVs, or a relatively flexible schedule such as refrigerators and air conditioners. For example, air conditioners can be turned on and off with different schedules as long as the room temperature is within a comfortable range.

\subsection{Converters for PV and Battery}
Note that the PV, battery, grid, and loads are all connected to an AC bus. Since PV generation is operated on DC, a DC-to-AC converter is necessary, and its efficiency is assumed to be a constant $\subscr{\eta}{pv}$ satisfying $$0 < \subscr{\eta}{pv} \leq 1~.$$

Since the battery is also operated on DC, an AC-to-DC converter is necessary when charging the battery, and a DC-to-AC converter is necessary when discharging as shown in Fig.~\ref{fig1}. For simplicity, we assume that both converters have the same constant conversion efficiency $\subscr{\eta}{B}$ satisfying $$0 < \subscr{\eta}{B} \leq 1~.$$ We define
\begin{equation}
\subscr{P}{BC}(t) = \begin{cases} \subscr{\eta}{B} P_B(t) & \mathrm{if~} P_B(t) < 0\\
\frac{P_B(t)}{\subscr{\eta}{B}} & \mathrm{otherwise}~,
\end{cases} \nonumber
\end{equation}
In other words, $\subscr{P}{BC}(t)$ is the power exchanged with the AC bus when the converters and the battery are treated as an entity. Similarly, we can derive
\begin{equation}
P_B(t) = \begin{cases} \frac{\subscr{P}{BC}(t)}{\subscr{\eta}{B}} & \mathrm{if~} \subscr{P}{BC}(t) < 0\\
\subscr{\eta}{B} \subscr{P}{BC}(t) & \mathrm{otherwise}~,
\end{cases} \nonumber
\end{equation}
Note that $\eta_B^2$ is the round trip efficiency of the battery storage.

\subsection{Cost Minimization} \label{section_problem}
With all the components introduced earlier, now we can formulate the
following problem of minimizing the sum of the net power purchase cost\footnote{Note that the net power purchase cost include the positive cost to purchase electricity from the grid, and the negative cost to sell electricity back to the grid.} and the cost associated with the battery capacity loss while guaranteeing that the demand from loads and the peak shaving requirement are satisfied:
\begin{align}
\min_{P_B, P_g} &\int_{t_0}^{t_0 + T} \subscr{C}{g}(\tau) \subscr{P}{g}(\tau) d \tau + K \Delta C(t_0 + T) \nonumber\\
\textrm{s.t.}~& \subscr{\eta}{pv} \subscr{P}{pv}(t) + \subscr{P}{g}(t) = \subscr{P}{BC}(t) + \subscr{P}{load}(t)~,\label{eq:power_balance}\\
& \frac{d E_B (t)}{d t} = P_B(t)~,\nonumber\\
& \frac{d \Delta C(t)}{d t} = \begin{cases} - Z P_B(t) & \mathrm{if~} P_B(t) < 0\\
0 & \mathrm{otherwise}~,
\end{cases} \nonumber\\
& 0 \leq E_B(t) \leq \subscr{C}{ref} - \Delta C(t)~,\nonumber\\
& E_B(t_0) = 0~, \Delta C(t_0) = 0~,\nonumber\\
& \subscr{P}{Bmin} \leq P_B(t) \leq \subscr{P}{Bmax}~,\nonumber\\
& \subscr{P}{Bmax} = -\subscr{P}{Bmin} = \frac{\subscr{C}{ref} - \Delta C(t)}{T_c}, \nonumber\\
& P_B(t) = \begin{cases} \frac{\subscr{P}{BC}(t)}{\subscr{\eta}{B}} & \mathrm{if~} \subscr{P}{BC}(t) < 0\\
\subscr{\eta}{B} \subscr{P}{BC}(t) & \mathrm{otherwise}~,
\end{cases} \nonumber\\
& \subscr{P}{g}(t) \leq D~,\label{optimization_onestep}
\end{align}
where $t_0$ is the initial time, $T$ is the time period considered for the cost minimization, $K (\$/Wh) > 0$ is the unit cost for the battery capacity loss. Note that:
\begin{itemize}
\item[i)] No cost is associated with PV generation. In other words, PV generated electricity is assumed free;
\item[ii)] $K \Delta C(t_0 + T)$ is the loss of the battery purchase investment during the time period from $t_0$ to $t_0 + T$ due to the use of the battery to reduce the net power purchase cost;
\item[iii)] Eq.~\eqref{eq:power_balance} is the power balance
requirement for any time $t \in [t_0, t_0 + T]$;
\item[iv)] the constraint $P_g(t) \leq D$ captures the peak shaving requirement.
\end{itemize}

Given a battery of initial capacity $\subscr{C}{ref}$, on the one hand, if the battery is rarely used, then the cost due to the capacity loss $K \Delta C(t_0 + T)$ is low while the net power purchase cost $\int_{t_0}^{t_0 + T} \subscr{C}{g}(\tau) \subscr{P}{g}(\tau) d \tau$ is high; on the other hand, if the battery is used very often, then the net power purchase cost $\int_{t_0}^{t_0 + T} \subscr{C}{g}(\tau) \subscr{P}{g}(\tau) d \tau$ is low while the cost due to the capacity loss $K \Delta C(t_0 + T)$ is high. Therefore, there is a tradeoff on the use of the battery, which is characterized by calculating an optimal control policy on $P_B(t), P_g(t)$ to the optimization problem in Eq.~\eqref{optimization_onestep}.

\begin{remark}
Besides the constraint $P_g(t) \leq D$, peak shaving is also accomplished indirectly through dynamic pricing. Time-of-use price margins and schedules are motivated by the peak load magnitude and timing. Minimizing the net power purchase cost results in battery discharge and reduction in grid purchase during peak times. If, however, the peak load for a customer falls into the off-peak time period, then the constraint on $P_g(t)$ limits the amount of electricity that can be purchased. Peak shaving capabilities of the constraint $P_g(t) \leq D$ and dynamic pricing will be illustrated in Section~\ref{section5}.B. \label{remark:D} \hfill $\blacksquare$
\end{remark}

\subsection{Storage Size Determination}
Based on Eq.~\eqref{eq:power_balance}, we obtain $$\subscr{P}{g}(t) = \subscr{P}{load}(t) - \subscr{\eta}{pv} \subscr{P}{pv}(t) + \subscr{P}{BC}(t)~.$$ Let $u(t) = \subscr{P}{BC}(t)$, then the optimization problem in Eq.~\eqref{optimization_onestep} can be rewritten as
\begin{align}
\min_{u} &\resizebox{.9\hsize}{!}{$\int_{t_0}^{t_0 + T} \subscr{C}{g}(\tau) (\subscr{P}{load}(\tau) - \subscr{\eta}{pv} \subscr{P}{pv}(\tau) + u(\tau)) d \tau + K \Delta C(t_0 + T)$} \nonumber\\
\textrm{s.t.}~& \frac{d E_B (t)}{d t} = \begin{cases} \frac{u(t)}{\subscr{\eta}{B}} & \mathrm{if~} u(t) < 0\\
\subscr{\eta}{B} u(t) & \mathrm{otherwise}~,
\end{cases} \nonumber\\
& \frac{d \Delta C(t)}{d t} = \begin{cases} - Z \frac{u(t)}{\subscr{\eta}{B}} & \mathrm{if~} u(t) < 0\\
0 & \mathrm{otherwise}~,
\end{cases} \nonumber\\
& E_B(t) \geq 0~, E_B(t_0) = 0~, \Delta C(t_0) = 0~,\nonumber\\
& E_B(t) + \Delta C(t) \leq \subscr{C}{ref}~, \nonumber\\
& \subscr{\eta}{B} u(t) T_c + \Delta C(t) \leq \subscr{C}{ref}~\mathrm{if~} u(t) > 0, \nonumber\\
& - \frac{u(t)}{\subscr{\eta}{B}} T_c + \Delta C(t) \leq \subscr{C}{ref}~\mathrm{if~} u(t) < 0, \nonumber\\
& \subscr{P}{load}(t) - \subscr{\eta}{pv} \subscr{P}{pv}(t) + u(t) \leq D~.\label{optimization_onestep_simplified}
\end{align}
Now it is clear that only $u(t)$ is an independent variable. We define the set of feasible controls as controls that guarantee all the constraints in the optimization problem in Eq.~\eqref{optimization_onestep_simplified}.

Let $J$ denote the objective function $$\resizebox{.95\hsize}{!}{$\min_{u} \int_{t_0}^{t_0 + T} \subscr{C}{g}(\tau) (\subscr{P}{load}(\tau) - \subscr{\eta}{pv}\subscr{P}{pv}(\tau) + u(\tau)) d \tau + K \Delta C(t_0 + T)$}~.$$ If we \emph{fix} the parameters $t_0, T, K, Z, T_c$, and $D$, $J$ is a function of $\subscr{C}{ref}$, which is denoted as $J(\subscr{C}{ref})$. If we increase $\subscr{C}{ref}$, intuitively $J$ will decrease though may not strictly decrease (this is formally proved in Proposition~\ref{prop:monotonic}) because the battery can be utilized to decrease the cost by
\begin{itemize}
\item[i)] storing extra electricity generated from PV or purchasing electricity from the grid when the time-of-use pricing is low, and
\item[ii)] supplying the load or selling back when the time-of-use pricing is high.
\end{itemize}
Now we formulate the following storage size determination problem.

\begin{problem} (\textbf{Storage Size Determination}) Given the optimization problem in Eq.~\eqref{optimization_onestep_simplified} with fixed $t_0, T, K, Z, T_c$, and $D$, determine a critical value $\subpscr{C}{ref}{c} \geq 0$ such that
\begin{itemize}
\item $\forall \subscr{C}{ref} < \subpscr{C}{ref}{c}$, $J(\subscr{C}{ref}) > J(\subpscr{C}{ref}{c})$, and
\item $\forall \subscr{C}{ref} \geq \subpscr{C}{ref}{c}$, $J(\subscr{C}{ref}) = J(\subpscr{C}{ref}{c})$.
\end{itemize} \label{problem:battery_size}
\end{problem}

One approach to calculate the critical value $\subpscr{C}{ref}{c}$ is that we first obtain an explicit expression for the function $J(\subscr{C}{ref})$ by solving the optimization problem in Eq.~\eqref{optimization_onestep_simplified} and then solve for $\subpscr{C}{ref}{c}$ based on the function $J$. However, the optimization problem in Eq.~\eqref{optimization_onestep_simplified} is difficult to solve due to the nonlinear constraints on $u(t)$ and $E_B(t)$, and the fact that it is hard to obtain analytical expressions for $\subscr{P}{load}(t)$ and $\subscr{P}{pv}(t)$ in reality. Even though it might be possible to find the optimal control using the minimum principle~\cite{yru_journal:Bryson_1975}, it is still hard to get an explicit expression for the cost function $J$. Instead, in the next section, we identify conditions under which the storage size determination problem results in non-trivial solutions (namely, $\subpscr{C}{ref}{c}$ is positive and finite), and then propose lower and upper bounds on the critical battery capacity $\subpscr{C}{ref}{c}$.

\section{Bounds on $\subpscr{C}{ref}{c}$} \label{section3}

Now we examine the cost minimization problem in Eq.~\eqref{optimization_onestep_simplified}. Since $\subscr{P}{load}(t) - \subscr{\eta}{pv}\subscr{P}{pv}(t) + u(t) \leq D$, or equivalently, $u(t) \leq D + \subscr{\eta}{pv}\subscr{P}{pv}(t) - \subscr{P}{load}(t)$, is a constraint that has to be satisfied for any $t \in [t_0, t_0 + T]$, there are scenarios in which either there is no feasible control or $u(t) = 0$ for $t \in [t_0, t_0 + T]$. Given $\subscr{P}{pv}(t), \subscr{P}{load}(t)$, and $D$, we define
\begin{align}
&\resizebox{.85\hsize}{!}{$S_1 = \{t \in [t_0, t_0 + T]~|~D + \subscr{\eta}{pv}\subscr{P}{pv}(t) - \subscr{P}{load}(t) < 0\}$}~, \label{eq:S1}\\
&\resizebox{.85\hsize}{!}{$S_2 = \{t \in [t_0, t_0 + T]~|~D + \subscr{\eta}{pv}\subscr{P}{pv}(t) - \subscr{P}{load}(t) = 0\}$}~, \label{eq:S2}\\
&\resizebox{.85\hsize}{!}{$S_3 = \{t \in [t_0, t_0 + T]~|~D + \subscr{\eta}{pv}\subscr{P}{pv}(t) - \subscr{P}{load}(t) > 0\}$}~. \label{eq:S3}
\end{align}
Note that\footnote{$C = A \oplus B$ means $C = A \cup B$ and $A \cap B = \emptyset$.} $S_1 \oplus S_2 \oplus S_3 = [t_0, t_0 + T]$. Intuitively, $S_1$ is the set of time instants at which the battery can only be discharged, $S_2$ is the set of time instants at which the battery can be discharged or is not used (i.e., $u(t) = 0$), and $S_3$ is the set of time instants at which the battery can be charged, discharged, or is not used.

\begin{proposition}
Given the optimization problem in Eq.~\eqref{optimization_onestep_simplified}, if
\begin{itemize}
\item[\textbf{i)}] $t_0 \in S_1$, or
\item[\textbf{ii)}] $S_3$ is empty, or
\item[\textbf{iii)}] $t_0 \notin S_1$ (or equivalently, $t_0 \in S_2 \cup S_3$), $S_3$ is nonempty, $S_1$ is nonempty, and $\exists t_1 \in S_1$, $\forall t_3 \in S_3$, $t_1 < t_3$,
\end{itemize}
then either there is no feasible control or $u(t) = 0$ for $t \in [t_0, t_0 + T]$. \label{prop:existence}
\end{proposition}

\begin{proof}
Now we prove that, under these three cases, either there is no feasible control or $u(t) = 0$ for $t \in [t_0, t_0 + T]$.
\begin{itemize}
\item[\textbf{i)}] If $t_0 \in S_1$, then $u(t_0) \leq D + \subscr{\eta}{pv}\subscr{P}{pv}(t_0) - \subscr{P}{load}(t_0) < 0$. However, since $E_B(t_0) = 0$, the battery cannot be discharged at time $t_0$. Therefore, there is no feasible control.

\item[\textbf{ii)}] If $S_3$ is empty, it means that $u(t_0) \leq D + \subscr{\eta}{pv}\subscr{P}{pv}(t_0) - \subscr{P}{load}(t_0) \leq 0$ for any $t \in [t_0, t_0 + T]$, which implies that the battery can never be charged. If $S_1$ is nonempty, then there exists some time instant when the battery has to be discharged. Since $E_B(t_0) = 0$ and the battery can never be charged, there is no feasible control. If $S_1$ is empty, then $u(t) = 0$ for any $t \in [t_0, t_0 + T]$ is the only feasible control because $E_B(t_0) = 0$.

\item[\textbf{iii)}] In this case, the battery has to be discharged at time $t_1$, but the charging can only happen at time instant $t_3 \in S_3$. If $\forall t_3 \in S_3$, $\exists t_1 \in S_1$ such that $t_1 < t_3$, then the battery is always discharged before possibly being charged. Since $E_B(t_0) = 0$, then there is no feasible control.
\end{itemize}
\end{proof}

Note that if the only feasible control is $u(t) = 0$ for $t \in [t_0, t_0 + T]$, then the battery is not used. Therefore, we impose the following assumption.

\begin{assumption}
In the optimization problem in Eq.~\eqref{optimization_onestep_simplified}, $t_0 \in S_2 \cup S_3$, $S_3$ is nonempty, and either
\begin{itemize}
\item $S_1$ is empty, or
\item $S_1$ is nonempty, but $\forall t_1 \in S_1$, $\exists t_3 \in S_3$, $t_3 < t_1$,
\end{itemize}
where $S_1, S_2, S_3$ are defined in Eqs.~\eqref{eq:S1},~\eqref{eq:S2},~\eqref{eq:S3}. \label{assumption}
\end{assumption}

Given Assumption~\ref{assumption}, there exists at least one feasible control. Now we examine how $J(\subscr{C}{ref})$ changes when $\subscr{C}{ref}$ increases.

\begin{proposition}
Consider the optimization problem in Eq.~\eqref{optimization_onestep_simplified} with fixed $t_0, T, K, Z, T_c$, and $D$. If $\subscr{C}{ref}^1 < \subscr{C}{ref}^2$, then $J(\subscr{C}{ref}^1) \geq J(\subscr{C}{ref}^2)$. \label{prop:monotonic}
\end{proposition}

\begin{proof}
Given $\subscr{C}{ref}^1$, suppose control $u^1(t)$ achieves the minimum cost $J(\subscr{C}{ref}^1)$ and the corresponding states for the battery charge and capacity loss are $E_B^1(t)$ and $\Delta C^1(t)$. Since
\begin{align*}
& E_B^1(t) + \Delta C^1(t) \leq \subscr{C}{ref}^1 < \subscr{C}{ref}^2~, \nonumber\\
& \subscr{\eta}{B}u^1(t) T_c + \Delta C^1(t) \leq \subscr{C}{ref}^1 < \subscr{C}{ref}^2~\mathrm{if} ~u^1(t) > 0, \nonumber\\
& - \frac{u^1(t)}{\subscr{\eta}{B}} T_c + \Delta C^1(t) \leq \subscr{C}{ref}^1 < \subscr{C}{ref}^2~\mathrm{if} ~u^1(t) < 0, \nonumber\\
& \subscr{P}{load}(t) - \subscr{\eta}{pv}\subscr{P}{pv}(t) + u^1(t) \leq D~,
\end{align*}
$u^1(t)$ is also a feasible control for problem~\eqref{optimization_onestep_simplified} with $\subscr{C}{ref}^2$, and results in the cost $J(\subscr{C}{ref}^1)$. Since $J(\subscr{C}{ref}^2)$ is the minimum cost over the set of all feasible controls which include $u^1(t)$, we must have $J(\subscr{C}{ref}^1) \geq J(\subscr{C}{ref}^2)$.
\end{proof}

In other words, $J$ is non-increasing with respect to the parameter $\subscr{C}{ref}$, i.e., $J$ is monotonically decreasing (though may not be strictly monotonically decreasing). If $\subscr{C}{ref} = 0$, then $0 \leq \Delta C(t) \leq \subscr{C}{ref} = 0$, which implies that $u(t) = 0$. In this case, $J$ has the largest value
\begin{equation}
\subscr{J}{max} = J(0) = \int_{t_0}^{t_0 + T} \subscr{C}{g}(\tau) (\subscr{P}{load}(\tau) - \subscr{\eta}{pv}\subscr{P}{pv}(\tau)) d \tau~. \label{eq:Jmax}
\end{equation}

Proposition~\ref{prop:monotonic} also justifies the storage size determination problem. Note that the critical value $\subpscr{C}{ref}{c}$ (as defined in Problem~\ref{problem:battery_size}) is unique as shown below.

\begin{proposition}
Given the optimization problem in Eq.~\eqref{optimization_onestep_simplified} with fixed $t_0, T, K, Z, T_c$, and $D$, $\subpscr{C}{ref}{c}$ is unique. \label{prop:uniqueness}
\end{proposition}

\begin{proof}
We prove it via contradiction. Suppose $\subpscr{C}{ref}{c}$ is not unique. In other words, there are two different critical values $\subscr{C}{ref}^{c_1}$ and $\subscr{C}{ref}^{c_2}$. Without loss of generality, suppose $\subscr{C}{ref}^{c_1} < \subscr{C}{ref}^{c_2}$. By definition, $J(\subscr{C}{ref}^{c_1}) > J(\subscr{C}{ref}^{c_2})$ because $\subscr{C}{ref}^{c_2}$ is a critical value, while $J(\subscr{C}{ref}^{c_1}) = J(\subscr{C}{ref}^{c_2})$ because $\subscr{C}{ref}^{c_1}$ is a critical value. A contradiction. Therefore, we must have $\subscr{C}{ref}^{c_1} = \subscr{C}{ref}^{c_2}$.
\end{proof}

Intuitively, if the unit cost for the battery capacity loss $K$ is higher (compared with purchasing electricity from the grid), then it might be preferable that the battery is not used at all, which results in $\subpscr{C}{ref}{c} = 0$, as shown below.

\begin{proposition}
Consider the optimization problem in Eq.~\eqref{optimization_onestep_simplified} with fixed $t_0, T, K, Z, T_c$, and $D$ under Assumption~\ref{assumption}.
\begin{itemize}
\item[\textbf{i)}] If $$K \geq \frac{(\max_{t} C_g(t) - \min_{t} C_g(t)) \subscr{\eta}{B}}{Z}~,$$ then $J(\subscr{C}{ref}) = \subscr{J}{max}$, which implies that $\subpscr{C}{ref}{c} = 0$;
\item[\textbf{ii)}] if $$K < \frac{(\max_{t} C_g(t) - \min_{t} C_g(t)) \subscr{\eta}{B}}{Z}~,$$ then \begin{align*}
J(\subscr{C}{ref}) \geq &\subscr{J}{max} - (\max_{t} C_g(t) - \min_{t} C_g(t) - \frac{KZ}{\subscr{\eta}{B}}) \times\\
&T \times  (D + \max_{t } (\subscr{\eta}{pv}\subscr{P}{pv}(t) - \subscr{P}{load}(t)))~,
\end{align*}
where $\max_t$ and $\min_t$ are calculated for $t \in [t_0, t_0 + T]$.
\end{itemize} \label{prop:K}
\end{proposition}

\begin{proof}
The cost function can be rewritten as
\begin{align}
J(\subscr{C}{ref}) = &\subscr{J}{max} + \int_{t_0}^{t_0 + T} \subscr{C}{g}(\tau) u(\tau) d \tau + K \Delta C(t_0 + T)~\nonumber\\
=&\subscr{J}{max} + \int_{t_0}^{t_0 + T} \subscr{C}{g}(\tau) u(\tau) d \tau + \nonumber\\
&K \int_{t_0}^{t_0+ T} - Z \frac{u(\tau)}{\subscr{\eta}{B}}|_{u(\tau)<0} d \tau~\nonumber\\
= &\subscr{J}{max} + J_+ + J_-~,\nonumber
\end{align}
where $$J_+ = \int_{t_0}^{t_0 + T} \subscr{C}{g}(\tau) u(\tau)|_{u(\tau) > 0} d \tau~,$$ and $$J_- = \int_{t_0}^{t_0 + T} (\subscr{C}{g}(\tau) - \frac{KZ}{\subscr{\eta}{B}}) u(\tau)|_{u(\tau) <0} d \tau~.$$ Note that $J_+ \geq 0$ because the integrand $\subscr{C}{g}(\tau) u(\tau)|_{u(\tau) > 0}$ is always nonnegative.

\textbf{i)} We first consider $K \geq \frac{(\max_{t} C_g(t) - \min_{t} C_g(t)) \subscr{\eta}{B} }{Z}$. There are two possibilities:
\begin{itemize}
\item $K \geq \frac{(\max_{t} C_g(t)) \subscr{\eta}{B}}{Z}$, or equivalently, $\max_{t} C_g(t) \leq \frac{KZ}{\subscr{\eta}{B}}$. Thus, for any $t \in [t_0, t_0 + T]$, $C_g(t) - \frac{KZ}{\subscr{\eta}{B}} \leq 0$, which implies that $J_- \geq 0$. Therefore, $J(\subscr{C}{ref}) \geq \subscr{J}{max}$.
\item $\frac{(\max_{t} C_g(t) - \min_{t} C_g(t))\subscr{\eta}{B} }{Z} \leq K < \frac{(\max_{t} C_g(t))\subscr{\eta}{B}}{Z}$. Let $A_1 = \max_t C_g(t) - \frac{KZ}{\subscr{\eta}{B}}$, then $A_1 > 0$. Let $A_2 = \min_t C_g(t)$, then $A_2 \geq 0$. Since $\frac{(\max_{t} C_g(t) - \min_{t} C_g(t))\subscr{\eta}{B} }{Z} \leq K$, $A_2 \geq A_1$. Now we have
    $J_+ \geq A_2 \int_{t_0}^{t_0 + T} u(\tau)|_{u(\tau) > 0} d \tau$,
    and $J_- \geq A_1 \int_{t_0}^{t_0 + T} u(\tau)|_{u(\tau) <0} d \tau$. Therefore,
\begin{align*}
J(\subscr{C}{ref}) \geq &\subscr{J}{max} + A_2 \int_{t_0}^{t_0 + T} u(\tau)|_{u(\tau) > 0} d \tau + \\
&A_1 \int_{t_0}^{t_0 + T} u(\tau)|_{u(\tau) <0} d \tau\\
= &\subscr{J}{max} + (A_2 - A_1) \int_{t_0}^{t_0 + T} u(\tau)|_{u(\tau) > 0} d \tau +\\
&A_1 \int_{t_0}^{t_0 + T} u(\tau) d \tau\\
\geq &\subscr{J}{max} + A_1 \int_{t_0}^{t_0 + T} u(\tau) d \tau~.
\end{align*}
Since $E_B(t_0 + T) = E_B(t_0) + \int_{t_0}^{t_0 + T} P_B(\tau) d \tau \geq 0$ and $E_B(t_0) = 0$, we have
\begin{align}
0 \leq &\int_{t_0}^{t_0 + T} P_B(\tau) d \tau \nonumber\\
= & \int_{t_0}^{t_0 + T} P_B(\tau) |_{P_B(\tau) > 0} d \tau + \int_{t_0}^{t_0 + T} P_B(\tau) |_{P_B(\tau) < 0} d \tau \nonumber\\
= & \int_{t_0}^{t_0 + T} u(\tau) \subscr{\eta}{B} |_{u(\tau) > 0} d \tau + \int_{t_0}^{t_0 + T} \frac{u(\tau)}{\subscr{\eta}{B}} |_{u(\tau) < 0} d \tau \nonumber\\
\leq &\int_{t_0}^{t_0 + T} u(\tau)|_{u(\tau) > 0} d \tau + \int_{t_0}^{t_0 + T} u(\tau) |_{u(\tau) < 0} d \tau \nonumber\\
= &\int_{t_0}^{t_0 + T} u(\tau) d \tau~. \label{eq:int_u}
\end{align}
Therefore, we have $J(\subscr{C}{ref}) \geq \subscr{J}{max}$.
\end{itemize}
In summary, if $K \geq \frac{(\max_{t} C_g(t) - \min_{t} C_g(t)) \subscr{\eta}{B} }{Z}$, we have $J(\subscr{C}{ref}) \geq \subscr{J}{max}$, which implies that $J(\subscr{C}{ref}) \geq J(0)$. Since $J(\subscr{C}{ref})$ is a non-increasing function of $\subscr{C}{ref}$, we also have $J(\subscr{C}{ref}) \leq J(0)$. Therefore, we must have $J(\subscr{C}{ref}) = J(0) = \subscr{J}{max}$ for $\subscr{C}{ref} \geq 0$, and $\subpscr{C}{ref}{c} = 0$ by definition.

\textbf{ii)} We now consider $K < \frac{(\max_{t} C_g(t) - \min_{t} C_g(t)) \subscr{\eta}{B}}{Z}$, which implies that $A_1 > A_2 \geq 0$. Then
\begin{align*}
J(\subscr{C}{ref}) \geq &\subscr{J}{max} + A_2 \int_{t_0}^{t_0 + T} u(\tau)|_{u(\tau) > 0} d \tau +\\
&A_1 \int_{t_0}^{t_0 + T} u(\tau)|_{u(\tau) <0} d \tau\\
= &\subscr{J}{max} + A_2 \int_{t_0}^{t_0 + T} u(\tau) d \tau +\\
&(A_1 - A_2) \int_{t_0}^{t_0 + T} u(\tau)|_{u(\tau) <0} d \tau~.
\end{align*}
Since $\int_{t_0}^{t_0 + T} u(\tau) d \tau \geq 0$ as argued in the proof to \textbf{i)},  $J(\subscr{C}{ref}) \geq \subscr{J}{max} + (A_1 - A_2) \int_{t_0}^{t_0 + T} u(\tau)|_{u(\tau) <0} d \tau$. Now we try to lower bound $\int_{t_0}^{t_0 + T} u(\tau)|_{u(\tau) <0} d \tau$. Note that $\int_{t_0}^{t_0 + T} u(\tau) d \tau \geq 0$ (as shown in Eq.~\eqref{eq:int_u}) implies that $$\int_{t_0}^{t_0 + T} u(\tau)|_{u(\tau) <0} d \tau \geq - \int_{t_0}^{t_0 + T} u(\tau)|_{u(\tau) > 0} d \tau~.$$ Given Assumption~\ref{assumption}, $S_3$ is nonempty, which implies that $D + \max_t (\subscr{\eta}{pv}\subscr{P}{pv}(t) - \subscr{P}{load}(t)) > 0$. Since $u(t) \leq D + \subscr{\eta}{pv}\subscr{P}{pv}(t) - \subscr{P}{load}(t) \leq D + \max_t (\subscr{\eta}{pv}\subscr{P}{pv}(t) - \subscr{P}{load}(t))$, $\int_{t_0}^{t_0 + T} u(\tau)|_{u(\tau) > 0} d \tau \leq T \times (D + \max_t (\subscr{\eta}{pv}\subscr{P}{pv}(t) - \subscr{P}{load}(t)))$, which implies that
\begin{equation}
\resizebox{.85\hsize}{!}{$\int_{t_0}^{t_0 + T} u(\tau)|_{u(\tau) <0} d \tau \geq - T \times (D + \max_t (\subscr{\eta}{pv}\subscr{P}{pv}(t) - \subscr{P}{load}(t)))$}~. \label{eq:lower_bound_int}
\end{equation}
In summary, $J(\subscr{C}{ref}) \geq \subscr{J}{max} - (A_1 - A_2) \times T \times (D + \max_t (\subscr{\eta}{pv}\subscr{P}{pv}(t) - \subscr{P}{load}(t)))$, which proves the result.
\end{proof}

\begin{remark}
Note that \textbf{i)} in Proposition~\ref{prop:K} provides a criterion for evaluating the economic value of batteries compared to purchasing electricity from the grid. In other words, only if the unit cost of the battery capacity loss satisfies $$K < \frac{(\max_{t} C_g(t) - \min_{t} C_g(t)) \subscr{\eta}{B} }{Z}~,$$ it is desirable to use battery storages. This criterion depends on the pricing signal (especially the difference between the maximum and the minimum of the pricing signal), the conversion efficiency of the battery converters, and the aging coefficient $Z$. \hfill $\blacksquare$
\end{remark}

Given the result in Proposition~\ref{prop:K}, we impose the following additional assumption on the unit cost of the battery capacity loss to guarantee that $\subpscr{C}{ref}{c}$ is positive.

\begin{assumption}
In the optimization problem in Eq.~\eqref{optimization_onestep_simplified}, $$K < \frac{(\max_{t} C_g(t) - \min_{t} C_g(t)) \subscr{\eta}{B} }{Z}~.$$ \label{assumption_price}
\end{assumption}

\vspace{-15pt} In other words, the cost of the battery capacity loss during operations is less than the potential gain expressed as the margin between peak and off-peak prices modified by the conversion efficiency of the battery converters and the battery aging coefficient. Since $J(\subscr{C}{ref})$ is a non-increasing function of $\subscr{C}{ref}$ and lower bounded by a finite value given Assumptions~\ref{assumption} and~\ref{assumption_price}, the storage size determination problem is well defined. Now we show lower and upper bounds on the critical battery capacity in the following proposition.

\begin{proposition}
Consider the optimization problem in Eq.~\eqref{optimization_onestep_simplified} with fixed $t_0, T, K, Z, T_c$, and $D$ under Assumptions~\ref{assumption} and~\ref{assumption_price}. Then $\subpscr{C}{ref}{lb} \leq \subpscr{C}{ref}{c} \leq \subpscr{C}{ref}{ub}$, where $$\subpscr{C}{ref}{lb} = \max(\frac{T_c}{\subscr{\eta}{B}} \times (\max_{t} (\subscr{P}{load}(t) - \subscr{\eta}{pv}\subscr{P}{pv}(t)) - D), ~0)~,$$
and $$\subpscr{C}{ref}{ub} = \max (\subscr{\eta}{B} T_c + \frac{Z T}{\subscr{\eta}{B}}, ~\subscr{\eta}{B} T) \times (D + \max_t (\subscr{\eta}{pv}\subscr{P}{pv}(t) - \subscr{P}{load}(t))~.$$ \label{prop:upperbound}
\end{proposition}

\vspace{-15pt} \begin{proof} We first show the lower bound via contradiction. Without loss of generality, we assume that $\frac{T_c}{\subscr{\eta}{B}} \times (\max_{t} (\subscr{P}{load}(t) - \subscr{\eta}{pv}\subscr{P}{pv}(t)) - D) > 0$ (because we require $\subpscr{C}{ref}{c} \geq 0$). Suppose $$\subpscr{C}{ref}{c} < \subpscr{C}{ref}{lb} = \frac{T_c}{\subscr{\eta}{B}} \times (\max_{t} (\subscr{P}{load}(t) - \subscr{\eta}{pv}\subscr{P}{pv}(t)) - D)~,$$ or equivalently,
$$D < -\frac{\subpscr{C}{ref}{c} \subscr{\eta}{B}}{T_c} + \max_t(\subscr{P}{load}(t) - \subscr{\eta}{pv}\subscr{P}{pv}(t))~.$$
Therefore, there exists $t_1 \in [t_0, t_0 + T]$ such that $D < -\frac{\subpscr{C}{ref}{c} \subscr{\eta}{B}}{T_c} + \subscr{P}{load}(t_1) - \subscr{\eta}{pv}\subscr{P}{pv}(t_1)$.  Since $u(t_1) \leq D - \subscr{P}{load}(t_1) + \subscr{\eta}{pv}\subscr{P}{pv}(t_1) < -\frac{\subpscr{C}{ref}{c} \subscr{\eta}{B}}{T_c} \leq 0$, we have  \begin{align*}
P_B(t_1) = \frac{u(t_1)}{\subscr{\eta}{B}} &\leq \frac{D + \subscr{\eta}{pv}\subscr{P}{pv}(t_1) - \subscr{P}{load}(t_1)}{\subscr{\eta}{B}} \\
&< -\frac{\subpscr{C}{ref}{c}}{T_c} \leq -\frac{\subscr{C}{ref} - \Delta C(t_1)}{T_c} = \subscr{P}{Bmin}~.
\end{align*}
The implication is that the control does not satisfy the discharging constraint at $t_1$. Therefore, $\subpscr{C}{ref}{lb} \leq \subpscr{C}{ref}{c}$.

To show $\subpscr{C}{ref}{c} \leq \subpscr{C}{ref}{ub}$, it is sufficient to show that if $\subscr{C}{ref} \geq \subpscr{C}{ref}{ub}$, the electricity that can be stored never exceeds $\subpscr{C}{ref}{ub}$.

Note that the battery charging is limited by $\subscr{P}{Bmax}$, i.e., $P_B(t) = \subscr{\eta}{B} u(t) \leq \subscr{P}{Bmax}$. Now we try to lower bound $\subscr{P}{Bmax}$. $$\subscr{P}{Bmax} = \frac{\subscr{C}{ref} - \Delta C(t)}{T_c} \geq \frac{\subscr{C}{ref} - \Delta C(t_0 + T)}{T_c}~,$$ in which the second inequality holds because $\Delta C(t)$ is a non-decreasing function of $t$. By applying the aging model in Eq.~\eqref{eq:battery_aging}, we have $$\resizebox{.95\hsize}{!}{$\subscr{P}{Bmax} \geq \frac{\subscr{C}{ref} - \Delta C(t_0 + T)}{T_c} = \frac{\subscr{C}{ref} + \int_{t_0}^{t_0 + T} Z \frac{u(\tau)}{\subscr{\eta}{B}}|_{u(\tau) <0} d \tau}{T_c}~.$}$$ Using Eq.~\eqref{eq:lower_bound_int}, we have $$\subscr{P}{Bmax} \geq \frac{\subscr{C}{ref} - \frac{Z T}{\subscr{\eta}{B}} \times (D + \max_t (\subscr{\eta}{pv}\subscr{P}{pv}(t) - \subscr{P}{load}(t)))}{T_c}~.$$ Since $\subscr{C}{ref} \geq \subpscr{C}{ref}{ub}$, we have
\begin{align*}
\subscr{P}{Bmax} \geq &\resizebox{.86\hsize}{!}{$\frac{(\max(\subscr{\eta}{B} T_c + \frac{Z T}{\subscr{\eta}{B}}, \subscr{\eta}{B} T) - \frac{Z T}{\subscr{\eta}{B}}) \times (D + \max_t (\subscr{\eta}{pv}\subscr{P}{pv}(t) - \subscr{P}{load}(t)))}{T_c}$}\\
\geq &\subscr{\eta}{B} (D + \max_t (\subscr{\eta}{pv}\subscr{P}{pv}(t) - \subscr{P}{load}(t)))~.
\end{align*}
Therefore, $u(t) \leq D + \subscr{\eta}{pv}\subscr{P}{pv}(t) - \subscr{P}{load}(t) \leq D + \max_t (\subscr{\eta}{pv}\subscr{P}{pv}(t) - \subscr{P}{load}(t)) \leq \frac{\subscr{P}{Bmax}}{\subscr{\eta}{B}}$, which implies that $P_B(t) = \subscr{\eta}{B} u(t) \leq \subscr{P}{Bmax}$. Thus, the only constraint on $u$ related to the battery charging is $u(t) \leq D + \max_t (\subscr{\eta}{pv}\subscr{P}{pv}(t) - \subscr{P}{load}(t))$. In turn, during the time interval $[t_0, t_0 + T]$, the maximum amount of electricity that can be charged is
\begin{align*}
&\int_{t_0}^{t_0 + T} P_B(\tau) |_{P_B(\tau) > 0} d\tau = \int_{t_0}^{t_0 + T} \subscr{\eta}{B} u(\tau) |_{u(\tau) > 0} d\tau\\
\leq &\subscr{\eta}{B} T \times (D + \max_t (\subscr{\eta}{pv}\subscr{P}{pv}(t) - \subscr{P}{load}(t)))~,
\end{align*} which is less than or equal to $\subpscr{C}{ref}{ub}$. In summary, if $\subscr{C}{ref} \geq \subpscr{C}{ref}{ub}$, the amount of electricity that can be stored never exceeds $\subpscr{C}{ref}{ub}$, and therefore, $\subpscr{C}{ref}{c} \leq \subpscr{C}{ref}{ub}$.
\end{proof}

\begin{remark}
As discussed in Proposition~\ref{prop:existence}, if $S_3$ is empty, or equivalently, $(D + \max_t (\subscr{\eta}{pv}\subscr{P}{pv}(t) - \subscr{P}{load}(t)) \leq 0$, then $\subpscr{C}{ref}{ub} \leq 0$, which implies that $\subpscr{C}{ref}{c} = 0$; this is consistent with the result in Proposition~\ref{prop:existence} when $S_1$ is empty.\hfill $\blacksquare$
\end{remark}

\section{Algorithms for Calculating $\subpscr{C}{ref}{c}$} \label{section4}

In this section, we study algorithms for calculating the critical battery capacity $\subpscr{C}{ref}{c}$.

Given the storage size determination problem, one approach to calculate the critical battery capacity is by calculating the function $J(\subscr{C}{ref})$ and then choosing the $\subscr{C}{ref}$ such that the conditions in Problem~\ref{problem:battery_size} are satisfied. Though $\subscr{C}{ref}$ is a continuous variable, we can only pick a finite number of $\subscr{C}{ref}$'s and then approximate the function $J(\subscr{C}{ref})$. One way to pick these values is that we choose $\subscr{C}{ref}$ from $\subpscr{C}{ref}{ub}$ to $\subpscr{C}{ref}{lb}$ with the step size\footnote{Note that one way to implement the critical battery capacity in practice is to connect multiple identical batteries of fixed capacity $\subscr{C}{fixed}$ in parallel. In this case, $\subscr{\tau}{cap}$ can be chosen to be $\subscr{C}{fixed}$.} $-\subscr{\tau}{cap} < 0$. In other words, $$\subpscr{C}{ref}{i} = \subpscr{C}{ref}{ub} - i \times \subscr{\tau}{cap}~,$$ where\footnote{The ceiling function $\lceil x \rceil$ is the smallest integer which is larger than or equal to $x$.} $i = 0, 1, ..., L$ and $L = \lceil \frac{\subpscr{C}{ref}{ub} - \subpscr{C}{ref}{lb}}{\subscr{\tau}{cap}} \rceil$. Suppose the picked value is $\subpscr{C}{ref}{i}$, and then we solve the optimization problem in Eq.~\eqref{optimization_onestep_simplified} with $\subpscr{C}{ref}{i}$. Since  the battery dynamics and aging model are nonlinear functions of $u(t)$, we introduce a binary indicator variable $I_{u(t)}$, in which $I_{u(t)} = 0$ if $u(t) \geq 0$ and $I_{u(t)} = 1$ if $u(t) < 0$. In addition, we use $\delta t$ as the sampling interval, and discretize Eqs.~\eqref{eq:battery_dynamic} and~\eqref{eq:battery_aging} as
\begin{align*}
E_B(k+1) &= E_B(k) + P_B(k) \delta t~,\\
\Delta C(k+1) &=\begin{cases} \Delta C(k) - Z P_B(k) \delta t & \mathrm{if~} P_B(k) < 0\\
\Delta C(k) & \mathrm{otherwise}~.
\end{cases}
\end{align*}
With the indicator variable $I_{u(t)}$ and the discretization of continuous dynamics, the optimization problem in Eq.~\eqref{optimization_onestep_simplified} can be converted to a mixed integer programming problem with indicator constraints (such constraints are introduced in CPlex~\cite{yru_journal:CPLEX}), and can be solved using the CPlex solver~\cite{yru_journal:CPLEX} to obtain $J(\subpscr{C}{ref}{i})$. In the storage size determination problem, we need to check if $J(\subpscr{C}{ref}{i}) = J(\subpscr{C}{ref}{c})$, or equivalently, $J(\subpscr{C}{ref}{i}) = J(\subpscr{C}{ref}{ub})$; this is because $J(\subpscr{C}{ref}{c}) = J(\subpscr{C}{ref}{ub})$ due to $\subpscr{C}{ref}{c} \leq \subpscr{C}{ref}{ub}$ and the definition of $\subpscr{C}{ref}{c}$. Due to numerical issues in checking the equality, we  introduce a small constant $\subscr{\tau}{cost} > 0$ so that we treat $J(\subpscr{C}{ref}{i})$ the same as $J(\subpscr{C}{ref}{ub})$ if $J(\subpscr{C}{ref}{i}) - J(\subpscr{C}{ref}{ub}) < \subscr{\tau}{cost}$. Similarly, we treat $J(\subpscr{C}{ref}{i}) > J(\subpscr{C}{ref}{ub})$ if $J(\subpscr{C}{ref}{i}) - J(\subpscr{C}{ref}{ub}) \geq \subscr{\tau}{cost}$. The detailed algorithm is given in Algorithm~\ref{algorithm_simple}. At Step~4, if $J(\subpscr{C}{ref}{i}) - J(\subpscr{C}{ref}{0}) \geq \subscr{\tau}{cost}$, or equivalently, $J(\subpscr{C}{ref}{i}) - J(\subpscr{C}{ref}{ub}) \geq \subscr{\tau}{cost}$, we have $J(\subpscr{C}{ref}{i}) > J(\subpscr{C}{ref}{ub})$. Because of the monotonicity property in Proposition~\ref{prop:monotonic}, we know $J(\subpscr{C}{ref}{j}) \geq J(\subpscr{C}{ref}{i}) > J(\subpscr{C}{ref}{i-1}) = J(\subpscr{C}{ref}{ub})$ for any $j = i+1, ..., L$. Therefore, the \textbf{for} loop can be terminated, and the approximated critical battery capacity is $\subpscr{C}{ref}{i-1}$.

It can be verified that Algorithm~\ref{algorithm_simple} stops after at most $L+1$ steps, or equivalently, after solving at most $$\lceil \frac{\subpscr{C}{ref}{ub} - \subpscr{C}{ref}{lb}}{\subscr{\tau}{cap}} \rceil + 1$$ optimization problems in Eq.~\eqref{optimization_onestep_simplified}. The accuracy of the critical battery capacity is controlled by the parameters $\delta t, \subscr{\tau}{cap}, \subscr{\tau}{cost}$. Fixing $\delta t, \subscr{\tau}{cost}$, the output is within $$[\subpscr{C}{ref}{c} - \subscr{\tau}{cap}, \subpscr{C}{ref}{c} + \subscr{\tau}{cap}]~.$$ Therefore, by decreasing $\subscr{\tau}{cap}$, the critical battery capacity can be approximated with an arbitrarily prescribed precision.

\begin{algorithm}[tb] \small
\caption{Simple Algorithm for Calculating $\subpscr{C}{ref}{c}$} \label{algorithm_simple}

\begin{algorithmic}

\REQUIRE The optimization problem in Eq.~\eqref{optimization_onestep_simplified} with fixed $t_0, T, K, Z, T_c, D$ under Assumptions~\ref{assumption} and~\ref{assumption_price}, the calculated bounds $\subpscr{C}{ref}{lb}, \subpscr{C}{ref}{ub}$, and parameters $\delta t, \subscr{\tau}{cap}, \subscr{\tau}{cost}$

\ENSURE An approximation of $\subpscr{C}{ref}{c}$ \vspace{2pt}

\end{algorithmic}

\begin{algorithmic}[1]

\STATE Initialize $\subpscr{C}{ref}{i} = \subpscr{C}{ref}{ub} - i \times \subscr{\tau}{cap}$, where $i = 0, 1, ..., L$ and $L = \lceil \frac{\subpscr{C}{ref}{ub} - \subpscr{C}{ref}{lb}}{\subscr{\tau}{cap}} \rceil$;

\FOR{$i = 0, 1, ..., L$}

\STATE Solve the optimization problem in Eq.~\eqref{optimization_onestep_simplified} with $\subpscr{C}{ref}{i}$, and obtain $J(\subpscr{C}{ref}{i})$;

\IF{$i \geq 1$ and $J(\subpscr{C}{ref}{i}) - J(\subpscr{C}{ref}{0}) \geq \subscr{\tau}{cost}$}

\STATE Set $\subpscr{C}{ref}{c} = \subpscr{C}{ref}{i-1}$, and exit the \textbf{for} loop;

\ENDIF

\ENDFOR

\STATE Output $\subpscr{C}{ref}{c}$.
\end{algorithmic}
\end{algorithm}

Since the function $J(\subscr{C}{ref})$ is a non-increasing function of $\subscr{C}{ref}$, we propose Algorithm~\ref{algorithm_efficient} based on the idea of bisection algorithms. More specifically, we maintain three variables $$\subpscr{C}{ref}{1} < \subpscr{C}{ref}{3} < \subpscr{C}{ref}{2}~,$$ in which $\subpscr{C}{ref}{1}$ (or $\subpscr{C}{ref}{2}$) is initialized as $\subpscr{C}{ref}{lb}$ (or $\subpscr{C}{ref}{ub}$). We set $\subpscr{C}{ref}{3}$ to be $\frac{\subpscr{C}{ref}{1} + \subpscr{C}{ref}{2}}{2}$. Due to Proposition~\ref{prop:monotonic}, we have $$J(\subpscr{C}{ref}{1}) \geq J(\subpscr{C}{ref}{3}) \geq J(\subpscr{C}{ref}{2})~.$$ If $J(\subpscr{C}{ref}{3}) = J(\subpscr{C}{ref}{2})$ (or equivalently, $J(\subpscr{C}{ref}{3}) = J(\subpscr{C}{ref}{ub})$; this is examined in Step~7), then we know that $\subpscr{C}{ref}{1} \leq \subpscr{C}{ref}{c} \leq \subpscr{C}{ref}{3}$; therefore, we update $\subpscr{C}{ref}{2}$ with $\subpscr{C}{ref}{3}$ but do not update\footnote{Note that if we update $J(\subpscr{C}{ref}{2})$ with $J(\subpscr{C}{ref}{3})$, then the difference between $J(\subpscr{C}{ref}{2})$ and $J(\subpscr{C}{ref}{ub})$ can be amplified when $\subpscr{C}{ref}{2}$ is updated again later on.} the value $J(\subpscr{C}{ref}{2})$. On the other hand, if $J(\subpscr{C}{ref}{3}) > J(\subpscr{C}{ref}{2})$, then we know that $\subpscr{C}{ref}{2} \geq \subpscr{C}{ref}{c} \geq \subpscr{C}{ref}{3}$; therefore, we update $\subpscr{C}{ref}{1}$ with $\subpscr{C}{ref}{3}$ and set $J(\subpscr{C}{ref}{1})$ with $J(\subpscr{C}{ref}{3})$. In this case, we also check if $\subpscr{C}{ref}{2} - \subpscr{C}{ref}{3} < \subscr{\tau}{cap}$: if it is, then output $\subpscr{C}{ref}{2}$ since we know the critical battery capacity is between $\subpscr{C}{ref}{3}$ and $\subpscr{C}{ref}{2}$; otherwise, the \textbf{while} loop is repeated. Since every execution of the \textbf{while} loop halves the interval $[\subpscr{C}{ref}{1}, \subpscr{C}{ref}{2}]$ starting from $[\subpscr{C}{ref}{lb}, \subpscr{C}{ref}{ub}]$, the maximum number of executions of the \textbf{while} loop is $\lceil \log_2 \frac{\subpscr{C}{ref}{ub} - \subpscr{C}{ref}{lb}}{\subscr{\tau}{cap}} \rceil$,
and the algorithm requires solving at most $$\lceil \log_2 \frac{\subpscr{C}{ref}{ub} - \subpscr{C}{ref}{lb}}{\subscr{\tau}{cap}} \rceil + 1$$ optimization problems in Eq.~\eqref{optimization_onestep_simplified}. This is in contrast to solving $\lceil \frac{\subpscr{C}{ref}{ub} - \subpscr{C}{ref}{lb}}{\subscr{\tau}{cap}} \rceil + 1$ optimization problems in Eq.~\eqref{optimization_onestep_simplified} using Algorithm~\ref{algorithm_simple}.

\begin{algorithm}[tb] \small
\caption{Efficient Algorithm for Calculating $\subpscr{C}{ref}{c}$} \label{algorithm_efficient}

\begin{algorithmic}

\REQUIRE The optimization problem in Eq.~\eqref{optimization_onestep_simplified} with fixed $t_0, T, K, Z, T_c, D$ under Assumptions~\ref{assumption} and~\ref{assumption_price}, the calculated bounds $\subpscr{C}{ref}{lb}, \subpscr{C}{ref}{ub}$, and parameters $\delta t, \subscr{\tau}{cap}, \subscr{\tau}{cost}$

\ENSURE An approximation of $\subpscr{C}{ref}{c}$ \vspace{2pt}

\end{algorithmic}

\begin{algorithmic}[1]

\STATE Let $\subpscr{C}{ref}{1} = \subpscr{C}{ref}{lb}$ and $\subpscr{C}{ref}{2} = \subpscr{C}{ref}{ub}$;

\STATE Solve the optimization problem in Eq.~\eqref{optimization_onestep_simplified} with $\subpscr{C}{ref}{2}$, and obtain $J(\subpscr{C}{ref}{2})$;

\STATE Let $\mathrm{sign} = 1$;

\WHILE{$\mathrm{sign} = 1$}

\STATE Let $\subpscr{C}{ref}{3} = \frac{\subpscr{C}{ref}{1} + \subpscr{C}{ref}{2}}{2}$;

\STATE Solve the optimization problem in Eq.~\eqref{optimization_onestep_simplified} with $\subpscr{C}{ref}{3}$, and obtain $J(\subpscr{C}{ref}{3})$;

\IF{$J(\subpscr{C}{ref}{3}) - J(\subpscr{C}{ref}{2}) < \subscr{\tau}{cost}$}

\STATE Set $\subpscr{C}{ref}{2} = \subpscr{C}{ref}{3}$, and $J(\subpscr{C}{ref}{2}) = J(\subpscr{C}{ref}{ub})$;

\ELSE

\STATE Set $\subpscr{C}{ref}{1} = \subpscr{C}{ref}{3}$ and set $J(\subpscr{C}{ref}{1})$ with $J(\subpscr{C}{ref}{3})$;

\IF{$\subpscr{C}{ref}{2} - \subpscr{C}{ref}{3} < \subscr{\tau}{cap}$}

\STATE Set $\mathrm{sign} = 0$;

\ENDIF

\ENDIF

\ENDWHILE

\STATE Output $\subpscr{C}{ref}{c} = \subpscr{C}{ref}{2}$.

\end{algorithmic}
\end{algorithm}

\begin{remark}
Note that the critical battery capacity can be implemented by connecting batteries with fixed capacity in parallel because we only assume that the minimum battery charging time is fixed. \hfill $\blacksquare$
\end{remark}

\section{Simulations} \label{section5}

In this section, we calculate the critical battery capacity using Algorithm~\ref{algorithm_efficient}, and verify the results in Section~\ref{section3} via simulations. The parameters used in Section~\ref{section2} are chosen based on typical residential home settings and commercial buildings.

\subsection{Setting}
The GHI data is the measured GHI in July 2010 at La Jolla, California. In our simulations, we use $\eta = 0.15$, and $S = 10 m^2$. Thus  $\subscr{P}{pv}(t) = 1.5 \times \mathrm{GHI}(t) (W)$. We have two choices for $t_0$:
\begin{itemize}
\item $t_0$ is $0000$ h local standard time (LST) on Jul 8, 2010, and the hourly PV output is given in Fig.~\ref{fig6}(a) for the following four days starting from $t_0$ (this corresponds to the scenario in which there are relatively large variations in the PV output);

\item $t_0$ is $0000$ h LST on Jul 13, 2010, and the hourly PV output is given in Fig.~\ref{fig6}(b) for the following four days starting from $t_0$ (this corresponds to the scenario in which there are relatively small variations in the PV output).
\end{itemize}

\begin{figure}[tb] \centering
\subfigure[Starting from July 8, 2011.]{\includegraphics[scale=0.28]{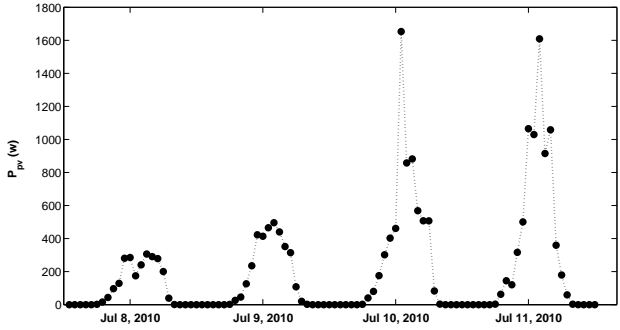}} \hfil
\subfigure[Starting from July 13, 2011.]{\includegraphics[scale=0.28]{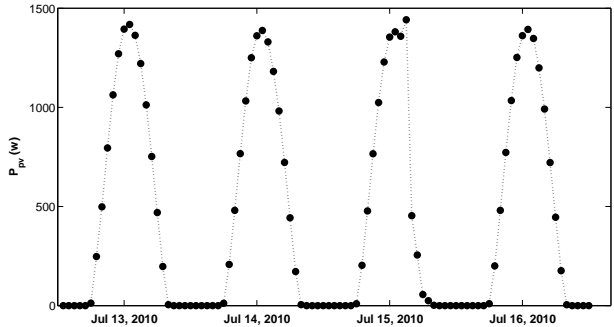}}
\caption{PV output based on GHI measurement at La Jolla, California, where tick marks indicate noon local standard time for each day.} \label{fig6}
\end{figure}

The time-of-use electricity purchase rate $\subscr{C}{g}(t)$ is
\begin{itemize}
\item $16.5$\textcent/$kWh$ from 11AM to 6PM (on-peak);
\item $7.8$\textcent/$kWh$ from 6AM to 11AM and 6PM to 10PM (semi-peak);
\item $6.1$\textcent/$kWh$ for all other hours (off-peak).
\end{itemize} This rate is for the summer season proposed by SDG\&E~\cite{yru_journal:DynamicPricing}, and is plotted in Fig.~\ref{fig12}.

\begin{figure}[tb] \centering
\includegraphics[scale=0.26]{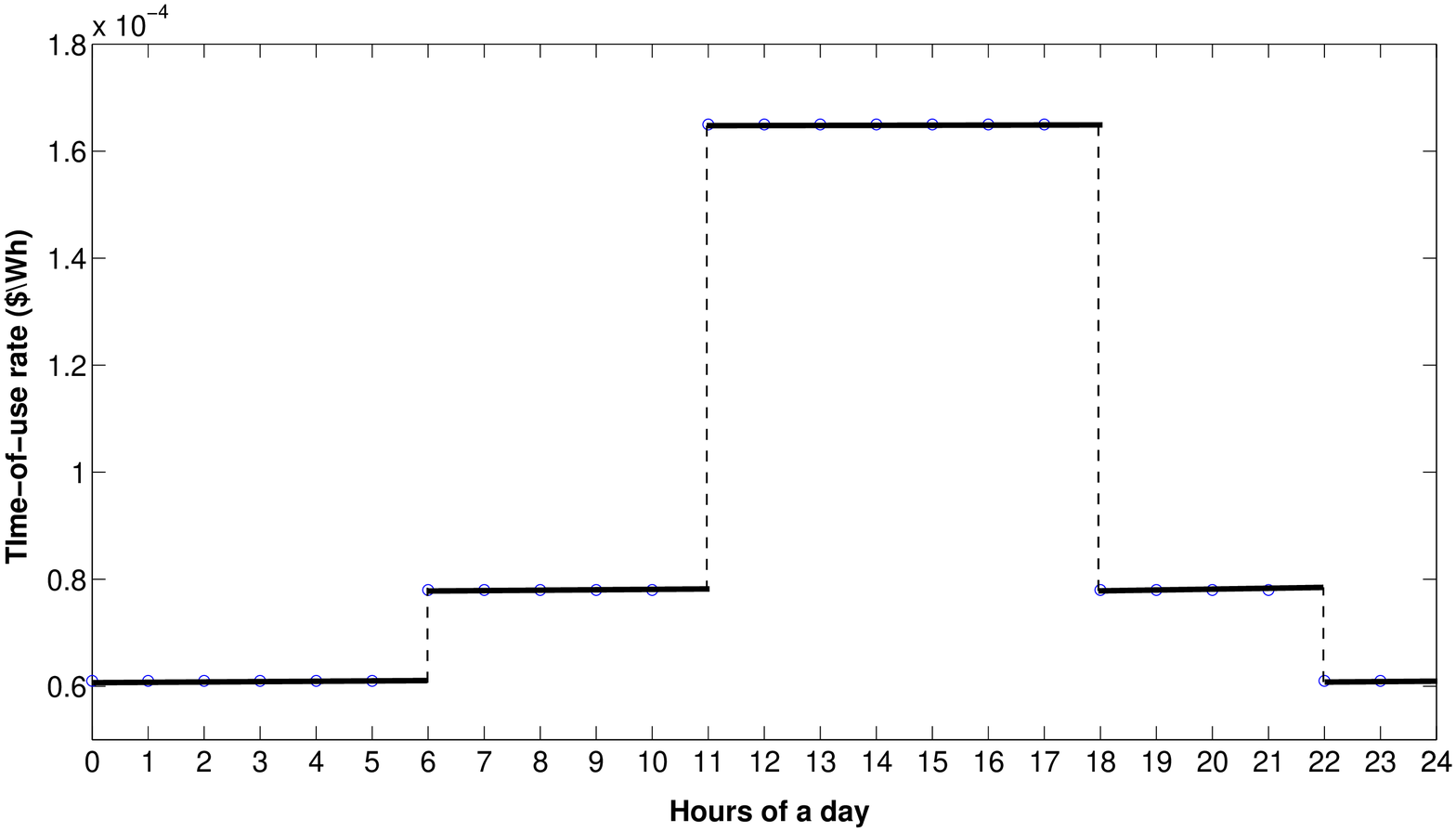}
\caption{Time-of-use pricing for the summer season at San Diego, CA~\cite{yru_journal:DynamicPricing}.} \label{fig12}
\end{figure}

Since the battery dynamics and aging are characterized by continuous ordinary differential equations, we use $\delta t = 1h$ as the sampling interval, and discretize Eqs.~\eqref{eq:battery_dynamic} and~\eqref{eq:battery_aging} as
\begin{align*}
E_B(k+1) &= E_B(k) + P_B(k)~,\\
\Delta C(k+1) &= \begin{cases} \Delta C(k) - Z P_B(k) & \mathrm{if~} P_B(k) < 0\\
\Delta C(k) & \mathrm{otherwise}~.
\end{cases}
\end{align*}
In simulations, we assume that lead-acid batteries are used; therefore, the aging coefficient is $Z = 3 \times 10^{-4}$~\cite{yru_journal:Riffonneau_2011}, the unit cost for capacity loss is $K = 0.15 \$/Wh$ based on the cost of $150\$/kWh$~\cite{yru_journal:Ton_2008}, and the minimum charging time is $T_c = 12 h$~\cite{yru_journal:Charging_leadacid}.

For the PV DC-to-AC converter and the battery DC-to-AC/AC-to-DC converters, we use $\subscr{\eta}{pv} = \subscr{\eta}{B} = 0.9$. It can be verified that Assumption~\ref{assumption_price} holds because the threshold value for $K$ is\footnote{Suppose the battery is a Li-ion battery with the same aging coefficient as a lead-acid battery. Since the unit cost $K = 1.333 \$/Wh$ based on the cost of $1333\$/kWh$~\cite{yru_journal:Ton_2008} (and $K = 0.78 \$/Wh$ based on the $10$-year projected cost of $780\$/kWh$~\cite{yru_journal:Ton_2008}), the use of such a battery is not as competitive as directly purchasing electricity from the grid.}
\begin{align*}
&\frac{(\max_{t} C_g(t) - \min_{t} C_g(t)) \subscr{\eta}{B} }{Z} = \\
&\frac{(16.5 - 6.1) \times 10^{-5} \times 0.9}{3 \times 10^{-4}} \approx 0.3120~.
\end{align*}

For the load, we consider two typical load profiles: the residential load profile as given in Fig.~\ref{fig10}(a), and the commercial load profile as given in Fig.~\ref{fig10}(b). Both profiles resemble the corresponding load profiles in Fig.~8 of~\cite{yru_journal:Rahman_2007}.\footnote{However, simulations in~\cite{yru_journal:Rahman_2007} start at $7$AM so Fig.~\ref{fig10} is a shifted version of the load profile in Fig.~8 in~\cite{yru_journal:Rahman_2007}.} Note that in the residential load profile, one load peak appears in the early morning, and the other in the late evening; in contrast, in the commercial load profile, the two load peaks appear during the daytime and occur close to each other. For multiple day simulations, the load is periodic based on the load profiles in Fig.~\ref{fig10}.

For the parameter $D$, we use $D = 800(W)$. Since the maximum of the loads in Fig.~\ref{fig10} is around $1000(W)$, we will illustrate the peak shaving capability of battery storage. It can be verified that Assumption~\ref{assumption} holds.

\begin{figure}[tb] \centering
\subfigure[Residential load averaged at $536.8(W)$.]{\includegraphics[scale=0.26]{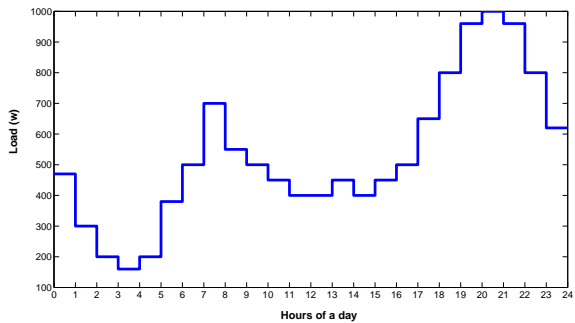}} \hfil
\subfigure[Commercial load averaged at $485.6(W)$.]{\includegraphics[scale=0.26]{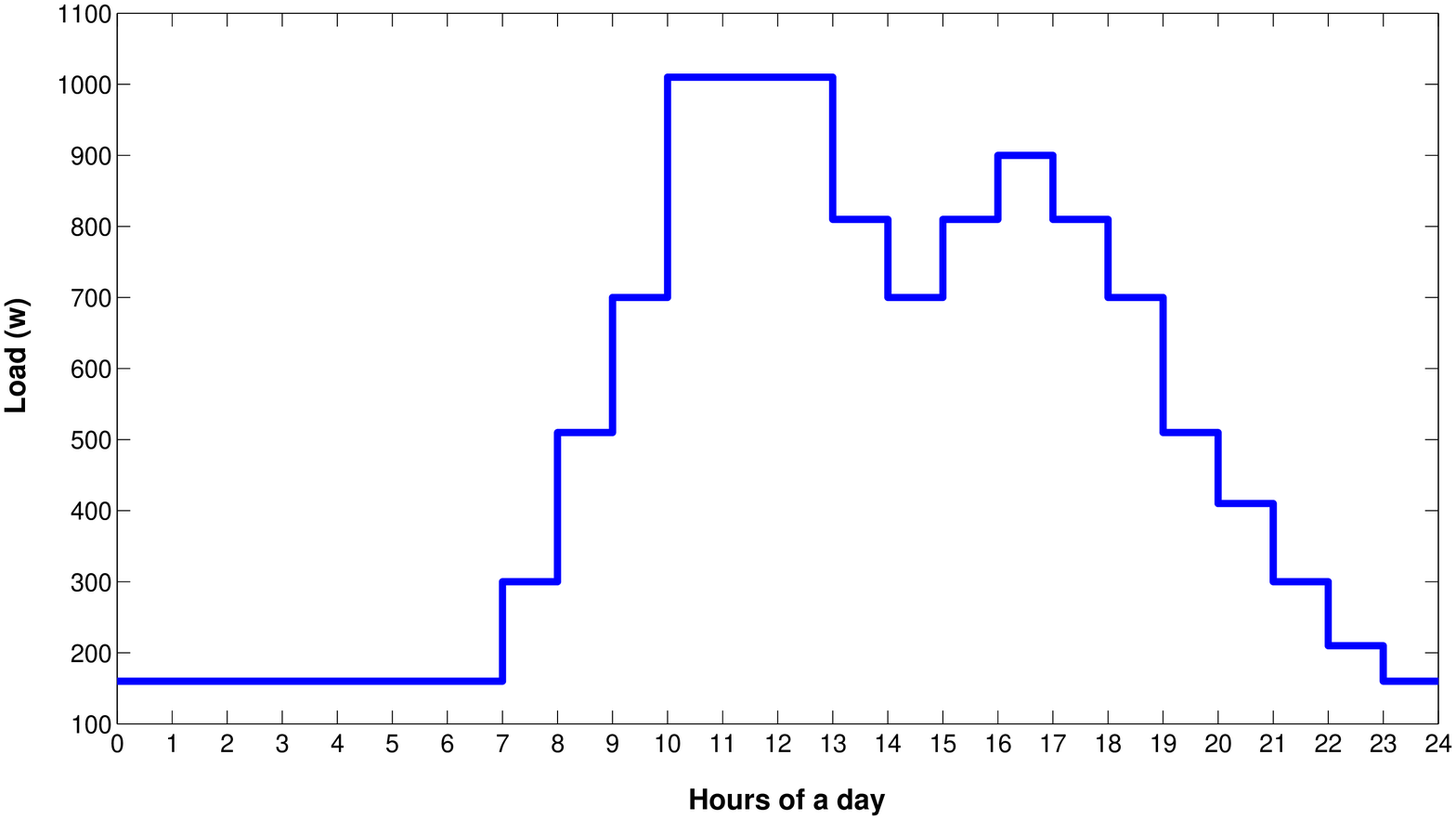}}
\caption{Typical residential and commercial load profiles.} \label{fig10}
\end{figure}

\subsection{Results}
We first examine the storage size determination problem using both Algorithm~\ref{algorithm_simple} and Algorithm~\ref{algorithm_efficient} in the following setting (called the basic setting):
\begin{itemize}
\item the cost minimization duration is $T = 24 (h)$,
\item $t_0$ is $0000$ h LST on Jul 13, 2010, and
\item the load is the residential load as shown in Fig.~\ref{fig10}(a).
\end{itemize}
The lower and upper bounds in Proposition~\ref{prop:upperbound} are calculated as $2667(Wh)$ and $39269(Wh)$. When applying Algorithm~\ref{algorithm_simple}, we choose $\subscr{\tau}{cap} = 10 (Wh)$ and $\subscr{\tau}{cost} = 10^{-4}$, and have $L = 3660$. When running the algorithm, $2327$ optimization problems in Eq.~\eqref{optimization_onestep_simplified} have been solved. The maximum cost $\subscr{J}{max}$ is $-0.1921$ while the minimum cost is $-0.3222$, which is larger than the lower bound $-2.5483$ as calculated based on Proposition~\ref{prop:K}. The critical battery capacity is calculated to be $16089(Wh)$. In contrast, only $$\lceil \log_2 \frac{\subpscr{C}{ref}{ub}}{\subscr{\tau}{cap}} \rceil + 1 = 13$$ optimization problems in Eq.~\eqref{optimization_onestep_simplified} are solved using Algorithm~\ref{algorithm_efficient} while obtaining the same critical battery capacity.

Now we examine the solution to the optimization problem in Eq.~\eqref{optimization_onestep_simplified} in the basic setting with the critical battery capacity $16089(Wh)$. We solve the mixed integer programming problem using the CPlex solver~\cite{yru_journal:CPLEX}, and the objective function is $J(16089) = -0.3222$. $P_B(t)$, $E_B(t)$ and $C(t)$ are plotted in Fig.~\ref{fig13}(a). The plot of $C(t)$ is consistent with the fact that there is capacity loss (i.e., the battery ages) only when the battery is discharged. The capacity loss is around $2.2 (Wh)$, and $$\frac{\Delta C(t_0 + T)}{\subscr{C}{ref}} = \frac{2.2}{16089} = 1.4 \times 10^{-4} \approx 0~,$$ which justifies the assumption we make when linearizing the nonlinear battery aging model in the Appendix. The dynamic pricing signal $C_g(t)$, $P_B(t)$, $\subscr{P}{load}(t)$ and $P_g(t)$ are plotted in Fig.~\ref{fig13}(b). From the second plot in Fig.~\ref{fig13}(b), it can be observed that the battery is charged when the time-of-use pricing is low in the early morning, and is discharged when the time-of-use pricing is high. From the third plot in Fig.~\ref{fig13}(b), it can be verified that, to minimize the cost, electricity is purchased from the grid when the time-of-use pricing is low, and is sold back to the grid when the time-of-use pricing is high; in addition, the peak demand in the late evening (that exceeds $D = 800(W)$) is shaved via battery discharging, as shown in detail in Fig.~\ref{fig13}(c).

\begin{figure}[tb] \centering
\subfigure[~]{\includegraphics[scale=0.28]{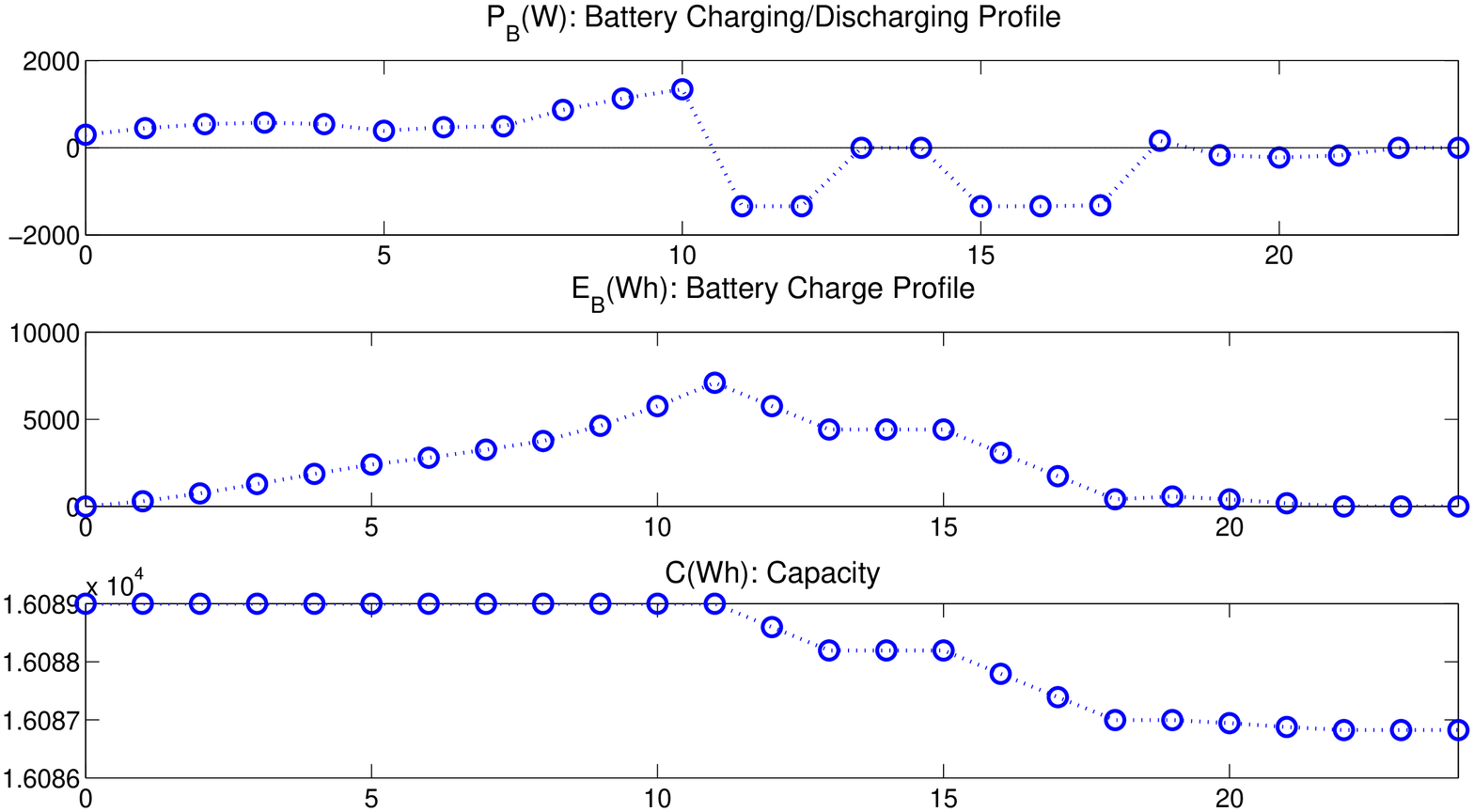}}
\subfigure[~]{\includegraphics[scale=0.28]{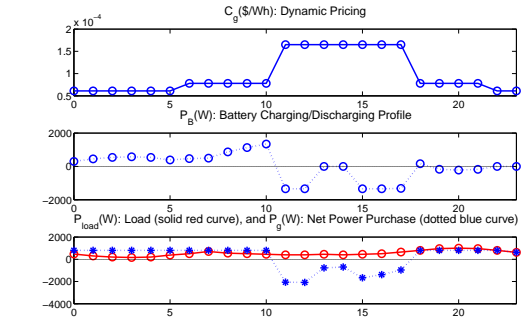}}
\subfigure[~]{\includegraphics[scale=0.28]{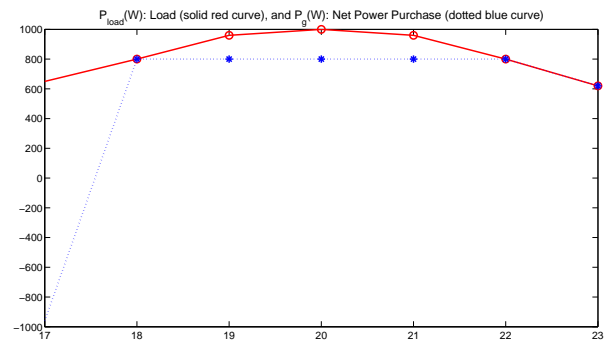}}
\caption{Solution to a typical setting in which $t_0$ is on Jul 13, 2010, $T = 24(h)$, the load is shown in Fig.~\ref{fig10}(a), and $\subscr{C}{ref} = 16089 (Wh)$.} \label{fig13}
\end{figure}

These observations also hold for $T = 48(h)$. In this case, we change $T$ to be $48(h)$ in the basic setting, and solve the battery sizing problem. The critical battery capacity is calculated to be $\subpscr{C}{ref}{c} = 16096(Wh)$. Now we examine the solution to the optimization problem in Eq.~\eqref{optimization_onestep_simplified} with the critical battery capacity $16096(Wh)$, and obtain $J(16096) = -0.6032$. $P_B(t)$, $E_B(t)$, $C(t)$ are plotted in Fig.~\ref{fig14}(a), and the dynamic pricing signal $C_g(t)$, $P_B(t)$, $\subscr{P}{load}(t)$ and $P_g(t)$ are plotted in Fig.~\ref{fig14}(b). Note that the battery is gradually charged in the first half of each day, and then gradually discharged in the second half to be empty at the end of each day, as shown in the second plot of Fig.~\ref{fig14}(a).
\begin{figure}[tb] \centering
\subfigure[~]{\includegraphics[scale=0.28]{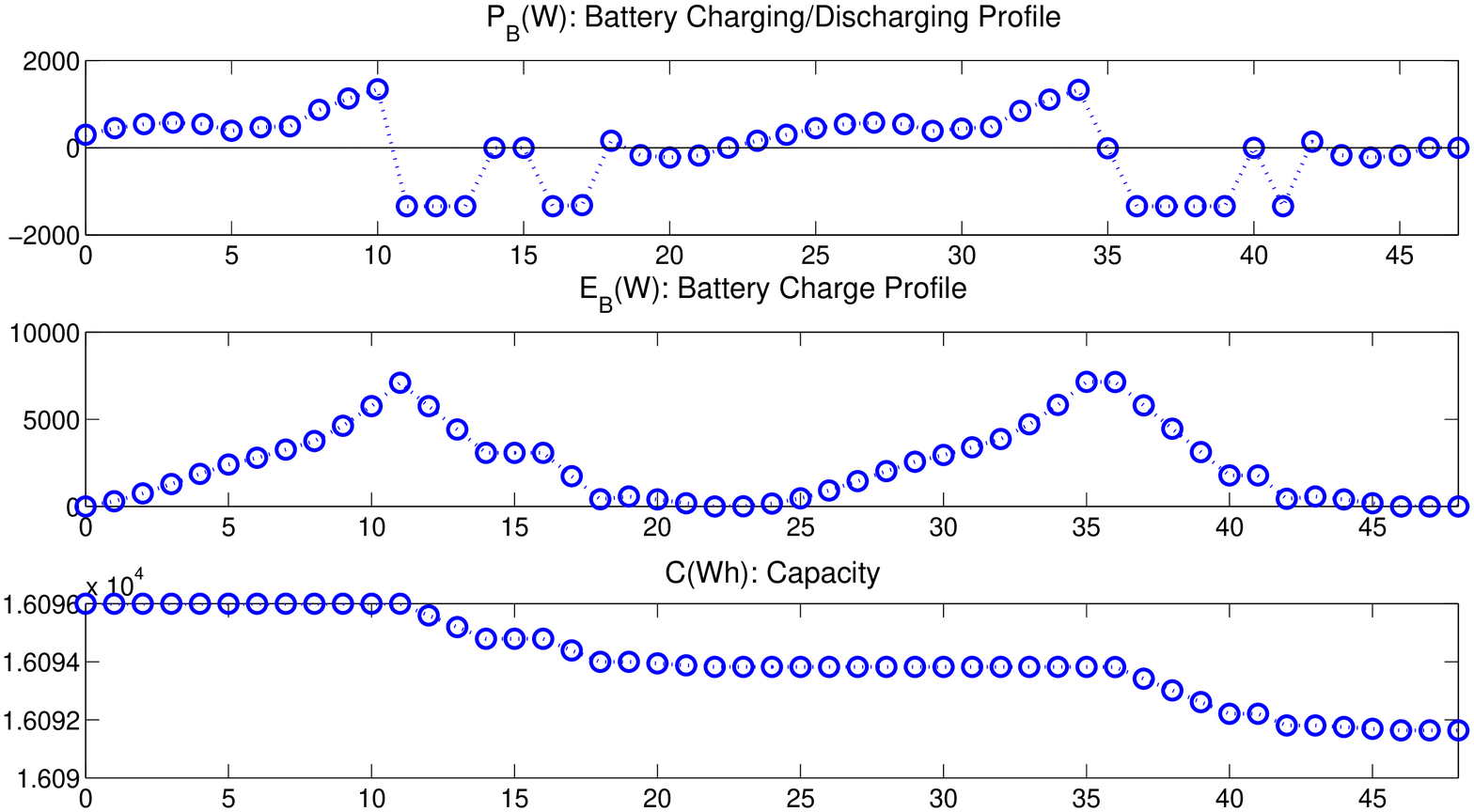}}
\subfigure[~]{\includegraphics[scale=0.28]{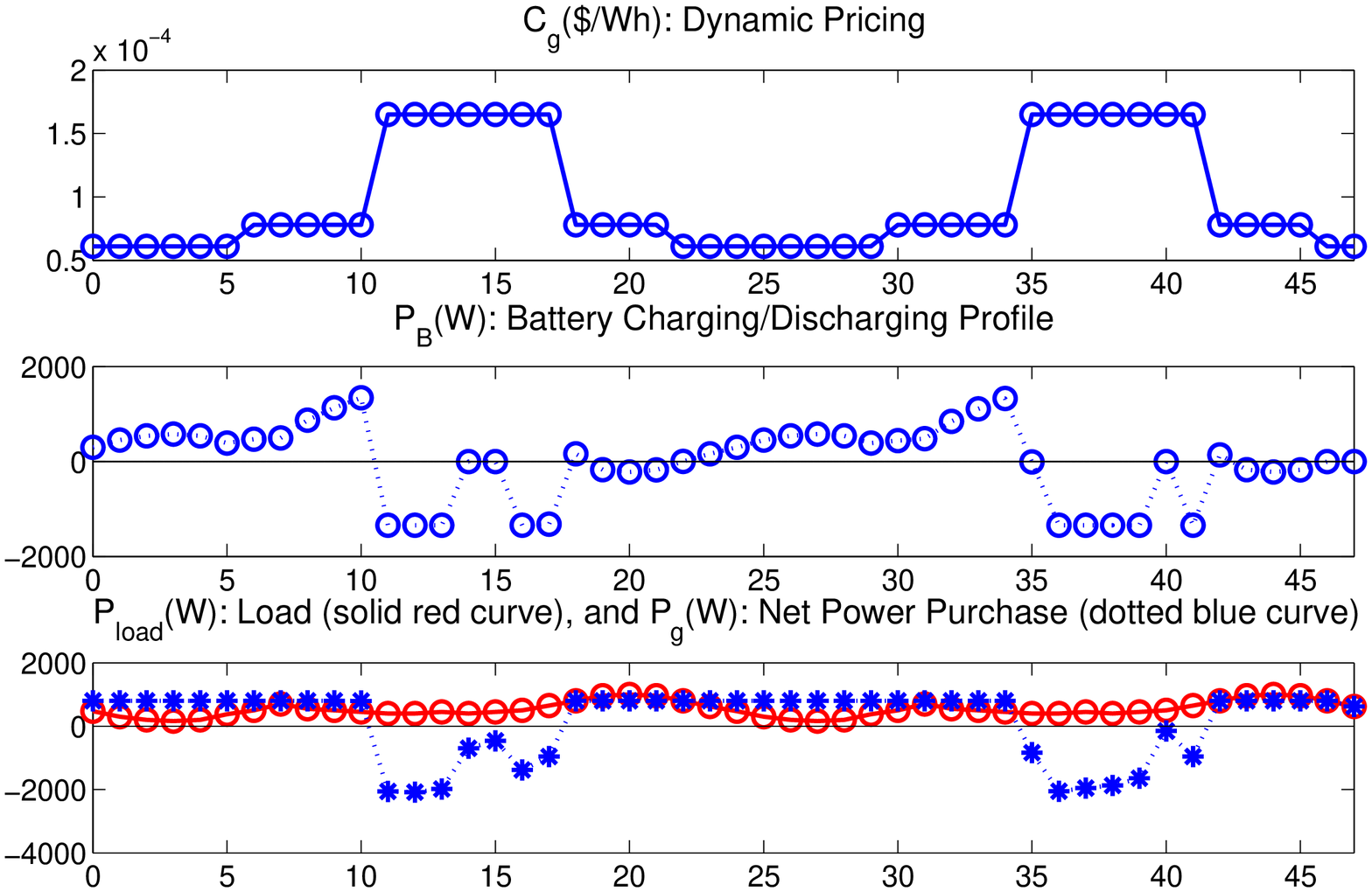}}
\caption{Solution to a typical setting in which $t_0$ is on Jul 13, 2010, $T = 48(h)$, the load is shown in Fig.~\ref{fig10}(a), and $\subscr{C}{ref} = 16096 (Wh)$.} \label{fig14}
\end{figure}

To illustrate the peak shaving capability of the dynamic pricing signal as discussed in Remark~\ref{remark:D}, we change the load to the commercial load as shown in Fig.~\ref{fig10}(b) in the basic setting, and solve the battery sizing problem. The critical battery capacity is calculated to be $\subpscr{C}{ref}{c} = 13352(Wh)$, and $J(13352) = -0.1785$. The dynamic pricing signal $C_g(t)$, $P_B(t)$, $\subscr{P}{load}(t)$ and $P_g(t)$ are plotted in Fig.~\ref{fig15}. For the commercial load, the duration of the peak loads coincides with that of the high price. To minimize the total cost, during peak times the battery is discharged, and the surplus electricity from PV after supplying the peak loads is sold back to the grid resulting in a negative net power purchase from the grid, as shown in the third plot in Fig.~\ref{fig15}. Therefore, unlike the residential case, the high price indirectly forces the shaving of the peak loads.
\begin{figure}[tb] \centering
\includegraphics[scale=0.28]{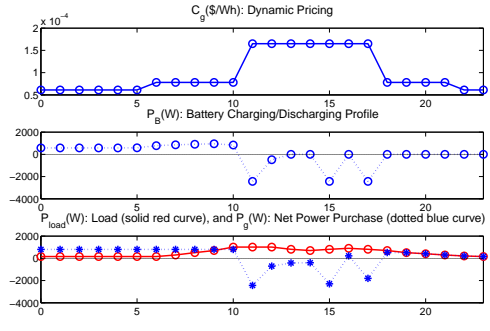}
\caption{Solution to a typical setting in which $t_0$ is on Jul 13, 2010, $T = 24(h)$, the load is shown in Fig.~\ref{fig10}(b), and $\subscr{C}{ref} = 13352 (Wh)$.} \label{fig15}
\end{figure}

Now we consider settings in which the load could be either residential loads or commercial loads, the starting time could be on either Jul 8 or Jul 13, 2010, and the cost optimization duration can be $24(h), 48(h), 96(h)$. The results are shown in Tables~\ref{tab1} and~\ref{tab2}. In Table~\ref{tab1}, $t_0$ is on Jul 8, 2010, while in Table~\ref{tab2}, $t_0$ is on Jul 13, 2010. In the pair $(24, R)$, $24$ refers to the cost optimization duration, and $R$ stands for residential loads; in the pair $(24, C)$, $C$ stands for commercial loads.

We first focus on the effects of load types. From Tables~\ref{tab1} and~\ref{tab2}, the commercial load tends to result in a higher cost (even though the average cost of the commercial load is smaller than that of the residential load) because the peaks of the commercial load coincide with the high price. When there is a relatively large variation of PV generation, commercial loads tend to result in larger optimum battery capacity $\subpscr{C}{ref}{c}$ as shown in Table~\ref{tab1}, presumably because the peak load occurs during the peak pricing period and reductions in PV production have to be balanced by additional battery capacity; when there is a relatively small variation of PV generation, residential loads tend to result in larger optimum battery capacity $\subpscr{C}{ref}{c}$ as shown in Table~\ref{tab2}. If $t_0$ is on Jul 8, 2010, the PV generation is relatively lower than the scenario in which $t_0$ is on Jul 13, 2010, and as a result, the battery capacity is smaller and the cost is higher. This is because it is more profitable to store PV generated electricity than grid purchased electricity. From Tables~\ref{tab1} and~\ref{tab2}, it can be observed that the cost optimization duration has relatively larger impact on the battery capacity for residential loads, and relatively less impact for commercial loads.

In Tables~\ref{tab1} and~\ref{tab2}, the row $\subscr{J}{max}$ corresponds to the cost in the scenario without batteries, the row $J(\subpscr{C}{ref}{c})$ corresponds to the cost in the scenario with batteries of capacity $\subpscr{C}{ref}{c}$, and the row \textit{Savings}\footnote{Note that in Table~\ref{tab2}, part of the costs are negative. Therefore, the word ``Earnings" might be more appropriate than ``Savings".} corresponds to $\subscr{J}{max} - J(\subpscr{C}{ref}{c})$. For Table~\ref{tab1}, we can also calculate the relative percentage of savings using the formula
$\frac{\subscr{J}{max} - J(\subpscr{C}{ref}{c})}{\subscr{J}{max}}$, and get the row \textit{Percentage}. One observation is that the relative savings by using batteries increase as the cost optimization duration increases. For example, when $T = 96(h)$ and the load type is residential, $16.10\%$ cost can be saved when a battery of capacity $18320 (Wh)$ is used;
when $T = 96(h)$ and the load type is commercial, $25.79\%$ cost can be saved when a battery of capacity $15747 (Wh)$ is used. This clearly shows the benefits of utilizing batteries in grid-connected PV systems. In Table~\ref{tab2}, only the absolute savings are shown since negative costs are involved.

\begin{table*}[htbp]
\caption{Simulation Results for $t_0$ on Jul 8, 2011} 
\centering 
\begin{tabular}{c | c | c | c | c | c | c } 
\hline \hline 
& (24, R) & (48, R) & (96, R) & (24, C) & (48, C) & (96, C)\\

\hline $\subpscr{C}{ref}{c}$ & 6891 & 7854 & 18320 & 8747 & 13089 & 15747\\
\hline $J(\subpscr{C}{ref}{c})$ & 0.8390 &  1.5287 & 1.9016 & 0.9802 &  1.7572 & 2.3060\\
\hline $\subscr{J}{max}$ &  0.9212 & 1.7027 & 2.2666 & 1.1314 & 2.1232 & 3.1076\\
\hline \textit{Savings} & 0.0822 & 0.1740 & 0.3650 & 0.1512 & 0.3660 & 0.8016\\
\hline \textit{Percentage} & 8.92\%  &  10.22\%  &  16.10\%  &  13.36\%  &  17.24\% & 25.79\%\\
\hline 
\end{tabular} \label{tab1}

\caption{Simulation Results for $t_0$ on Jul 13, 2011} 
\centering 
\begin{tabular}{c | c | c | c | c | c | c } 
\hline \hline 
& (24, R) & (48, R) & (96, R) & (24, C) & (48, C) & (96, C)\\

\hline $\subpscr{C}{ref}{c}$ & 16089 & 16096 & 70209 & 13352 & 17363 & 32910\\
\hline $J(\subpscr{C}{ref}{c})$ & -0.3222 &  -0.6032 & -0.9930 & -0.1785 &  -0.3721 & -0.5866\\
\hline $\subscr{J}{max}$ &  -0.1921 & -0.3395 & -0.4649 & 0.0181 & 0.0810 & 0.3761\\
\hline \textit{Savings} & 0.1301  &  0.2637  &  0.5281  &  0.1966  &  0.4531  &  0.9627\\
\hline 
\end{tabular}
\label{tab2}
\end{table*}

\section{Conclusions} \label{section6}
In this paper, we studied the problem of determining the size of battery storage for grid-connected PV systems. We proposed lower and upper bounds on the storage size, and introduced an efficient algorithm for calculating the storage size. In our analysis, the conversion efficiency of the PV DC-to-AC converter and the battery DC-to-AC and AC-to-DC converters is assumed to be a constant. We acknowledge that this is not the case in the current practice, in which the efficiency of converters depends on the input power in a nonlinear fashion~\cite{yru_journal:Riffonneau_2011}. This will be part of our future work. In addition, we would like to extend the results to distributed renewable energy storage systems, and generalize our setting by taking into account stochastic PV generation.
%

\appendix

\noindent \textbf{Derivation of the Simplified Battery Aging Model}

Eqs.~(11) and (12) in~\cite{yru_journal:Riffonneau_2011} are used to model the battery capacity loss, and are combined and rewritten below using the notation in this work:
\begin{equation}
C(t + \delta t) - C(t) = - \subscr{C}{ref} \times Z \times (\frac{E_B(t)}{C(t)} - \frac{E_B(t + \delta t)}{C(t + \delta t)})~.\label{eq:1}
\end{equation}
If $\delta t$ is very small, then $C(t + \delta t) \approx C(t)$. Therefore, $\frac{E_B(t)}{C(t)} - \frac{E_B(t + \delta t)}{C(t + \delta t)} \approx \frac{- P_B(t) \delta t}{C(t)}$. If we plug in the approximation, divide $\delta t$ on both sides of Eq.~\eqref{eq:1}, and let $\delta t$ goes to $0$, then we have
\begin{equation}
\frac{d C(t)}{dt} = \subscr{C}{ref} \times Z \times \frac{P_B(t)}{C(t)}~.\label{eq:2}
\end{equation}
This holds only if $P_B(t) < 0$ as in~\cite{yru_journal:Riffonneau_2011}, i.e., there could be capacity loss only when discharging the battery.

Since Eq.~\eqref{eq:2} is a nonlinear equation, it is difficult to solve. Let $\Delta C(t) = \subscr{C}{ref} - C(t)$, then Eq.~\eqref{eq:2} can be rewritten as
\begin{equation}
\frac{d \Delta C(t)}{dt} = - Z \times \frac{P_B(t)}{1- \frac{\Delta C(t)}{\subscr{C}{ref}}}~.\label{eq:3}
\end{equation}
If $t$ is much shorter than the life time of the battery, then the percentage of the battery capacity loss $\frac{\Delta C(t)}{\subscr{C}{ref}}$ is very close to $0$. Therefore, Eq.~\eqref{eq:3} can be simplified to the following linear ODE
\begin{equation*}
\frac{d \Delta C(t)}{dt} = - Z \times P_B(t)~,
\end{equation*}
when $P_B(t) < 0$. If $P_B(t) \geq 0$, there is no capacity loss, i.e., $$\frac{d C(t)}{dt} = 0~.$$



%
\IEEEtriggeratref{18}
\bibliographystyle{IEEEtran}
\bibliography{IEEEabrv,yru_journal}

\end{document}